\documentclass[onefignum,onetabnum]{siamart171218}
\usepackage{multirow}
\usepackage{float}
\usepackage{amsmath}
\usepackage{amsfonts}
\usepackage{paralist}
\usepackage{amsfonts}
\usepackage{enumerate}
\usepackage{amssymb}
\usepackage{array}
\usepackage{listings}
\usepackage{color}
\usepackage{mathtools}
\usepackage{setspace}
\usepackage{algorithm}
\usepackage{algorithmic}
\usepackage{varwidth}
\usepackage{blkarray}
\usepackage{enumitem}
\usepackage{times}
\usepackage{amscd,amsmath,amsfonts}
\usepackage{setspace}
\usepackage{subfigure}
\usepackage{graphics}
\usepackage{graphicx,epsfig}
\usepackage{psfrag}
\usepackage{color}
\usepackage{tabularx}
\usepackage{array}
\usepackage[left=3.0cm,top=3.1cm,right=3.0cm,bottom=3.1cm,bindingoffset=0.5cm]{geometry}

\newtheorem{lema}[theorem]{Lemma}
\newcommand{\cF}{{\cal F}}

\newcommand{\cX}{{\cal X}}

\newcommand{\Prob}{\mathbf{P}}
\newcommand{\Exp}[1]{\mathbf{E}\left[#1\right]}
\newcommand{\ExpJ}[1]{\mathbf{E}_J\left[#1\right]}
\newcommand{\Expi}[1]{\mathbf{E}_i\left[#1\right]}

\newcommand{\diag}{\text{diag}}
\providecommand{\rset}[1]{\mathbb{R}^}

\providecommand{\norm}[1]{\lVert#1\rVert}

\usepackage[colorinlistoftodos,bordercolor=orange,backgroundcolor=orange!20,linecolor=orange,textsize=scriptsize]{todonotes}

\headers{Faster randomized block Kaczmarz algorithms}{I. Necoara}
\title{Faster randomized block Kaczmarz algorithms
\thanks{Submitted to the editors DATE.\funding{This work  was supported by the Executive Agency for Higher Education, Research and Innovation Funding (UEFISCDI), Romania,  PNIII-P4-PCE-2016-0731, project ScaleFreeNet, no. 39/2017. The author thanks Yu. Nesterov and F. Glineur from Universite Catholique de Louvain for useful discussions on the Chebyshev-based Kaczmarz scheme.}}}
\author{Ion Necoara\thanks{Department of Automatic Control and Systems Engineering, University Politehnica Bucharest, Splaiul Independentei 313, Bucharest, 060042, Romania   (\email{ion.necoara@acse.pub.ro}). }}

\begin{document}
\maketitle

\begin{abstract}
The  Kaczmarz algorithm is a simple iterative scheme for solving consistent linear systems.  At each step,  the method projects the current iterate onto the solution space of a single  constraint. Hence, it requires very low cost per iteration and storage, and it has a linear  rate of convergence. Distributed implementations of Kaczmarz have become, in recent years,  the \textit{de facto} architectural choice for large-scale linear systems. Therefore, in this paper  we develop a family of  randomized block  Kaczmarz algorithms that uses at each step a subset of the constraints and extrapolated  stepsizes, and can be deployed on distributed computing units.  Our approach is based on several new ideas and tools, including stochastic selection rule for the blocks of rows,   stochastic conditioning of the linear system, and novel strategies for designing extrapolated stepsizes.  We prove that  randomized block Kaczmarz algorithm converges linearly in expectation, with a rate  depending on the geometric properties of the matrix  and its submatrices and on the size of the blocks.  Our convergence analysis reveals that the algorithm is most effective when it is given a good sampling of  the rows into well-conditioned blocks. Besides providing a general framework for the design and analysis of randomized block Kaczmarz methods, our results resolve an open problem in the literature related to the theoretical understanding of observed practical efficiency of extrapolated block Kaczmarz methods.
\end{abstract}



%%%%%%%%%%%%%%%%%%%%%%%%%%%%%%%%%%%%%%%%%%%%

\section{Introduction}
Given a real matrix $A \in \rset^{m \times n}$ and a real vector $b \in \rset^m$, in this paper we   search for    a solution of the  linear system $Ax = b$:
\begin{align}
\label{LSP}
\text{Find} \;\; x \quad \text{s.t.} \quad  Ax = b.
\end{align}
We  assume throughout the paper that the system is consistent, that is there exists a vector $x^* \in \rset^n$ for which $Ax^* = b$.  Let us denote the set of solutions by $\cX=\{x \in \rset^n: \; Ax=b\}$.  Linear systems represent a modeling paradigm for solving many engineering and physics problems: partial differential equations \cite{OlsTyr:14}, sensor networks \cite{XiaBoy:05},  filtering \cite{KhaMou:07}, signal processing \cite{HanNie:90}, computerized tomography  \cite{Hou:73}, machine learning and optimal control \cite{PatNec:17}.  In these applications it is usually sufficient to  find a point which is not too far from the solution set $\cX$. In particular, one chooses the error tolerance $\varepsilon>0$ and aims to find a point $x$ satisfying $\|x - \Pi_{\cX}(x)\|^2  \leq \varepsilon$, where $\Pi_{\cX}(\cdot) = \arg \min_{y \in \cX} \|\cdot -y\|$ is the projection function onto solution set $\cX$, and $\|\cdot\|$ is the standard Euclidean norm on $\rset^n$. In the case when a randomized algorithm is used to find $x$, which  renders  $x$ a random vector, one replace this condition with $\Exp{ \|x - \Pi_{\cX}(x)\|^2 } \leq \varepsilon$, where $\Exp{\cdot}$ denotes the expectation with respect to the randomness of the algorithm.

\subsection{Iterative methods}
In  practice, $m$ and $n$ are  usually    large  so that  iterative  methods,  e.g.  the  so-called  row-action  methods   are  preferred  (in a row-action method only one block of rows of  $A$  is  used  in a certain iteration \cite{Cen:81}). One  of these  methods  is  the  iterative  method  of  Kaczmarz \cite{Kac:37,StrVer:09,LevLew:10}.  In some situations, it is even more efficient than the conjugate gradient method, which is the most popular iterative algorithm for solving large linear systems \cite{OlsTyr:14}.  In  fact   Kaczmarz algorithm   was  implemented  by  Hounsfield  in  the  very  first  medical  scanner \cite{Hou:73}. At each step,  the Kaczmarz algorithm projects the current iterate onto the solution space of a single  row $a_{i_k}^T$ and then choose the next iterate along the line connecting the current iterate and the projection, leading to the following iterative process:
\begin{align}
\label{eq:iter-k}
x^{k+1} = x^k - \alpha_k \frac{a_{i_k}^T x^k - b_{i_k}}{\|a_{i_k}\|^2} a_{i_k}.
\end{align}
Usually, the stepsize $\alpha_k$ is chosen  in the interval $(0, \; 2)$. For $\alpha_k=1$ we recover the basic Kaczmarz algorithm \cite{Kac:37}.  Note that this update rule  requires low cost per iteration  and storage of order ${\cal O}(n)$. In contrast, in {\em block} Kaczmarz methods a subset  of rows  $A_{J_k}$ are used at each iteration, with $J_k \subseteq [m]$ and $|J_k| >1$.  We usually distinguish two approaches. The first  variant is simply a block generalization of basic Kaczmarz algorithm, that is,  we project the current iterate onto the solution space of  the \textit{entire} block $A_{J_k}$ and then choose the next iterate along the line connecting the current iterate and the projection:
\begin{align}
\label{eq:fullK}
x^{k+1} = x^k - \alpha_k  A_{J_k}^\dagger (A_{J_k} x^k - b_{J_k}),
\end{align}
where $A_{J_k}^\dagger$ denotes the pseudoinverse of $A_{J_k}$. Usually, the stepsize $\alpha_k$ is chosen $1$. This is the approach followed e.g. in \cite{Elf:80,GowRic:15,NeeTro:14a,RicTak:17} and we refer to this iterative process as \textit{block projection}   Kaczmarz algorithm. The main drawback of  \eqref{eq:fullK} is that each iteration is expensive, since we need to apply  the pseudoinverse to a vector, or equivalently, we must solve a least-squares  problem  at  each  iteration, having cost per iteration of order ${\cal O}(\tau^2 n)$, where $\tau=|J_k|$. Moreover, it is not adequate for distributed implementations. The second variant of block Kaczmarz  avoids these issues,    by  projecting the current estimate onto  \textit{ each individual} row that forms the block matrix $A_{J_k}$, and the resulting projections are averaged to form the next iterate. This leads to the following iteration:
\begin{align}
\label{eq:iter-bk}
x^{k+1} = x^k - \alpha_k \left( \sum_{i \in J_k} \omega_i \frac{a_{i}^T x^k - b_{i}}{\|a_{i}\|^2} a_{i} \right),
\end{align}
where  the weights $\omega_i \!\in\! [0, \ 1]$ such that $\sum_{i \in J_k} \omega_i \!=\!1$, and $\alpha_k \!\in\! (0, \; 2)$.  Note that  update  \eqref{eq:iter-bk} is very easy to implement on distributed computing units and it is comparable in terms of  cost per iteration with the basic Kaczmarz update \eqref{eq:iter-k}, i.e., of order ${\cal O}(\tau n)$.  This is the scheme considered  e.g. in \cite{BauCom:06,Cen:81,Mer:63,Pie:84} and  we also analyze it  in this paper and refer to it as  \textit{block Kaczmarz algorithm}.  Assuming  $\alpha_k \in (0, \; 2)$, then the iterative process  \eqref{eq:iter-k} is known to converge linearly \cite{LevLew:10,StrVer:09} (see also Section \ref{sec:convbasicK}). Moreover, linear convergence results for the iteration \eqref{eq:fullK}, with particular stepsize  $\alpha_k=1$, were recently derived in \cite{GowRic:15,NeeTro:14a,RicTak:17}.  {\em However,  we are not aware of any convergence rates depending on the  size of the blocks $|J_k|$ and  the geometric properties of the matrix $A$ and its submatrices $A_{J_k}$ for the iterative process~\eqref{eq:iter-bk}}.

\subsection{Extrapolation}
\noindent It is well known that the  practical performance of block Kaczmarz method \eqref{eq:iter-bk}  can be enhanced, and often dramatically so, using {\em extrapolation}. This refers to the practice of moving {\em further} along the line connecting the last iterate and the average of the projections by using a stepsize  $\alpha_k \geq 2$, see e.g. \cite{BauCom:06}.  For example, since the iterative process \eqref{eq:iter-bk}  can be slow,  in \cite{Mer:63,Pie:84} an extrapolated  variant of  \eqref{eq:iter-bk} has been introduced  with the following \textit{adaptive} choice for the stepsize $\alpha_k$:
\begin{align}
\label{eq:extrapol_alp}
 \alpha_k =   \frac{ 2 \sum_{i \in J_k}  \bar{\omega}_i (a_{i}^T x^k - b_{i})^2  }{\|  \sum_{i \in J_k}  \bar{\omega}_i (a_{i}^T x^k - b_{i}) a_i \|^2},
\end{align}
where  we use the notation $\bar{\omega}_i = {\omega}_i / \|a_i\|^2$ and, for convenience, we define $0/0 = 1$. From Jensen's inequality it follows that $\alpha_k \geq 2$. However, in numerical experiments, it has been observed that the extrapolation parameter $\alpha_k$ from \eqref{eq:extrapol_alp} can be much larger than $2$. Moreover,   the sequence $x^k$ generated by the iterative process  \eqref{eq:iter-bk} using the extrapolated  adaptive stepsize $\alpha_k$ from \eqref{eq:extrapol_alp} usually converges much faster than the same sequence  $x^k$ from \eqref{eq:iter-bk}  but generated with stepsize $\alpha_k \in (0, \ 2)$ \cite{BauCom:06,Cen:81,CenChe:12,Mer:63,Pie:84}.  However, despite more than $80$ years of research on block Kaczmarz methods, the empirical success of extrapolation schemes is not  supported by theory. {\em That is, to the best of our knowledge, there is no theory explaining why these methods with $\alpha_k \geq 2$ require less iterations than their non-extrapolated variants  $\alpha_k=1$}.

\subsection{Rows importance}
While selecting the index set $J \subseteq [m]$ uniformly random appears as the most natural choice, it is likely the  case that some blocks of  rows of $A$ are more important than others. As an illustration, consider the scenario in which there exists  $T \subset [m]$ such that  $\cX = \{x \in \rset^n:   A_T x = b_T  \}$, where $A_T$ denotes the block matrix of $A$ whose rows are indexed in the set $T$. Clearly, the rows  $a_i$ for $i \in T$ are more important than the rows $a_i$ for $i \notin T$. This is an extreme scenario: if $T$ is known, one should simply remove the non-important rows  from the representation to begin with. However, even if none of the rows can be removed, it is often the case that some (blocks of) rows are more important than others in the sense that one should project on these  more often. In fact, the operator theory shows that some sampling strategies of the blocks of rows are more effective than others, in terms of conditioning, see e.g.  \cite{NeeTro:14a,Tro:11}. {\em We are not aware of any paper on block Kaczmarz method \eqref{eq:iter-bk} that take importance of blocks of rows into consideration.}  An exception to this are some recent works \cite{NeeTro:14a,GowRic:15,RicTak:17}, but on  the block projection   Kaczmarz algorithm \eqref{eq:fullK} (i.e., \cite{NeeTro:14a,GowRic:15,RicTak:17} analyze  rows importance for the method that projects the current estimate on the entire solution space of $A_Jx=b_J$, as opposed to our algorithm  \eqref{eq:iter-bk},  where we only project on the individual rows of the submatrix $A_J$ and then average).

\subsection{Outline} In Section~\ref{sec:contributions} we summarize  selected key contributions of this paper. In Section~\ref{sec:prelim} we present some preliminary results for Kaczmarz algorithm. In Section \ref{sec:RBK}  we define general  random block Kaczmarz   algorithms and derive new  convergence rates. In Section \ref{sec:KC} we present an acceleration of block  Kaczmarz algorithm using Chebyshev-based  stepsizes and derive the corresponding convergence rates.  
%Finally, in  Section  \ref{sec:num} we corroborate our theoretical results through numerical experiments.

\subsection{Notation}  For $x \in \rset^n$, the standard Euclidean norm is denoted by $\|x\| = \sqrt{x^T x}$.  For a positive integer $m$, let $[m] = \{1, 2, \dots, m \}$. By $e_i$  we denote the $i$th column of  the identity matrix $I_n  \in \rset^{n \times n}$.  Let $A \in \rset^{m \times n}$ be a matrix. By $\|A\|_F$, $\|A\|$, $\text{rank}(A)$, $a_i^T$, $\lambda_{\min}^{\text{nz}}(A)$ and $\lambda_{\max}(A)$ we denote its Frobenius norm, spectral norm,  rank, $i$th row,  the smallest non-zero eigenvalue, and the largest eigenvalue,  respectively. For an index set $J \subset [m]$, by $A_J \in \rset^{|J| \times n}$ we denote the  matrix with the rows $a_i^T$ for $ i \in J$.  The projection of a point $x$ onto  a closed convex set $X$ is denoted by  $\Pi_X(x) = \arg \min_z \{\|x-z\| : z \in  X\}$. A matrix is called \textit{normalized} if all its rows have the Euclidean norm equal to $1$.

%%%%%%%%%%%%%%%%%%%%%%%%%%%%%%%%%%%%%%%5
%%%%%%%%%%%%%%%%%%%%%%%%%%%%%%%%%%%%%%

\section{Contributions}
\label{sec:contributions}
In this section we briefly review our key contributions and results, leaving the theoretical  details to the rest of the paper.

\subsection{General framework}  We {\em develop a  unified framework for studying extrapolation and rows importance questions for consistent linear systems, together with randomized block Kaczmarz  methods for solving such systems of linear equalities}.  We define a probability space $([m], \cF, \Prob)$.   By sampling $J \sim \Prob$, we are {\em choosing} a block of rows $A_J$ from the matrix $A$. In this way we achieve two goals at the same time:
\begin{itemize}
\item[(i)] First, this sampling defines a general  {\em stochastic selection rule} which we shall use to design a {\em randomized  block Kaczmarz method}, described in Section~\ref{subsec:algs} below.
\item[(ii)] Second, the choice of probability measure is a natural way to assign {\em importance} to the blocks of  $A$.
\end{itemize}
Note that the probability $\Prob$ is a {\em parameter}  playing the dual  role of controlling the  representation of the solution set  $\cX$ as an intersection of blocks of rows of matrix $A$, and defining the importance sampling procedure, which in turn defines the algorithm.   For matrices with normalized rows (i.e. each row has norm $1$), we have identified the following \textit{stochastic conditioning} parameter:
\begin{align}
\label{eq:steplbloc}
\lambda_{\max}^{\text{block}} = \max_{J \sim \textbf{P}}  \lambda_{\max}(A_J^TA_J)
\end{align}
as  the key quantity characterizing  importance sampling. In particular, our  analysis reveals that the  most effective  importance rule is the one that makes  $\lambda_{\max}^{\text{block}}$ small, i.e.    there is a sampling of the blocks of  the rows into well-conditioned blocks.  Moreover, the operator theory literature provides detailed information about the existence and construction of such good sampling (see Section \ref{sec:pave}).

\subsection{Algorithms}
\label{subsec:algs}
We propose a block Kaczmarz algorithmic framework that uses a randomized  scheme to choose a subset  of the constraints at each iteration (see Sections \ref{sec:RBK}  and \ref{sec:KC}):
\begin{align*}
\text{(RBK)}: \qquad
\begin{split}
& \text{Draw at each step a sample} \;   J_k \sim \textbf{P} \; \text{and  update:}  \\
& x^{k+1} = x^k -  	\alpha_k  \left(\sum\limits_{i \in J_k} \omega_i^{k} \frac{a_i^T x^k - b_i}{\|a_i\|^2} a_i \right),
\end{split}
\end{align*}
where the weights  satisfy $\omega_i^k \in [0, \ 1]$ such that $\sum_{i \in J_k} \omega_i^k=1$.  One important property of  our algorithmic framework  is the use of several   \emph{extrapolated}  stepsizes $\alpha_k$,  that, in general,  are much larger than the  stepsize $\alpha_k \in (0, \ 2)$ usually used in the literature.  More precisely, we analyze three choices for the stepsize  $\alpha_k$: (i) one depending on  the geometric properties of the submatrices of $A$ of the form ${\cal O}(1/\lambda_{\max}^{\text{block}})$; (ii) one  adaptive  stepsize similar to \eqref{eq:extrapol_alp};  (iii) one  stepsize   using  the roots of the Chebyshev polynomials.  All three extrapolation procedures yield $\alpha_k \geq 2$ and hence they accelerate drastically the convergence of RBK algorithm. 
%(see also the numerical results of Section \ref{sec:num}).
Another feature of our algorithm is that it allows to project in \emph{parallel} onto several rows, thus providing flexibility in matching the implementation of the algorithm on the distributed architecture at hand.   Moreover, RBK algorithm  can be  interpreted, for some particular choices of the weights and stepsize,  as a {\em minibatch}  stochastic gradient descent or {\em block} coordinate descent method applied to a specific optimization problem.

%%%%%%%%%%%%%%%%%%%%%%%%%%%%%%%%%5

\subsection{Convergence rates}

\begin{table}
\label{tbl:complexity}
\begin{center}
\begin{tabular}{|p{3.5cm}|p{7cm}|p{2cm}|}
\hline
RBK algorithm  &   Convergence rates  &   Remarks  \\ \hline
constant  stepsize &
\multirow{2}{*}{ $\Exp{ \|x^k - x^*_k\|^2 } \leq  \left(  1 -  \frac{\tau}{m} \frac{ \lambda_{\min}^{\text{nz}}}{\lambda_{\max}^{\text{block}}}  \right)^k \! \| x^0 - x_{0}^*\|^2$}
&  \multirow{2}{*} {Theorem~\ref{th1:convergence}} \\
normalized $A$ &   &    \\ \hline
adaptive  stepsize &
\multirow{2}{*}{ $\Exp{ \|x^k - x^*_k\|^2 } \leq   \left(  1 -  \frac{\tau}{m} \frac{ \lambda_{\min}^{\text{nz}}}{\lambda_{\max}^{\text{block}}}  \right)^k \!  \| x^0 - x_{0}^*\|^2$}
&  \multirow{2}{*} {Theorem~\ref{th2:convergence}} \\
normalized  $A$ &   &    \\ \hline
Chebyshev  stepsize &
\multirow{2}{*}{ $\| \Exp{x^k - x^*_k} \|^2 \leq  \left( 1 - \sqrt{\frac{\lambda_{\min}}{\lambda_{\max}}} \right)^{2k} \! \| x^0 - x_{0}^*\|^2$}
&  \multirow{2}{*} {Theorem~\ref{th3:convergence}} \\
 normalized $A$ \!\&\! $\lambda_{\min}  \!>\!0$ &   &    \\ \hline
\end{tabular}
\caption{The key convergence results obtained in this paper for algorithm RBK for the  three choices of the extrapolated stepsize.  Here, matrix $A$ is normalized and $\lambda_{\max}$ and $\lambda_{\min} (\lambda_{\min}^{\text{nz}})$  denote the largest and  smallest (non-zero)  eigenvalue of $A A^T$, respectively.}
\end{center}
\end{table}

\noindent To the best of our knowledge, convergence rates of  Kaczmarz type methods were only previously derived  for  stepsizes belonging to the interval $(0, \ 2)$ \cite{GowRic:15,NeeTro:14a,RicTak:17,StrVer:09}. Moreover,  the existing convergence estimates for block Kaczmarz algorithm \eqref{eq:iter-bk} do not show any dependence on the size of the blocks $|J|$   or on  the geometric properties of the block submatrices $A_J$ \cite{BauCom:06,Cen:81,CenChe:12,Mer:63,Pie:84}. On the other hand, our convergence analysis for the randomized block   Kaczmarz  (RBK) algorithm  is one of the first proving an (expected) linear rate of convergence that is expressed explicitly in terms of the geometric properties of the matrix and its submatrices and of the size of the blocks.  Moreover, our analysis allows to derive convergence estimates  for  all three choices of the extrapolated stepsize. {\em From our knowledge, this is the first time the randomized block Kaczmarz algorithm  with extrapolation  ($|J|>1$ and $\alpha_k >2$) is shown to have a better convergence rate than its basic variant  \eqref{eq:iter-k} ($|J|=1$ and $\alpha_k=1$).}  We have identified $\lambda_{\max}^{\text{block}}$ as the key quantity determining whether extrapolation helps  or not, and how much (the smaller $\lambda_{\max}^{\text{block}}$, the more it helps).  For example, for normalized matrices,   RBK with the extrapolation rules (i)--(ii) has an expected linear rate   for the square distance of the iterates to the optimal solution set of the form (see  Table~\ref{tbl:complexity}):
\[ {\cal O} \left( \tfrac{m \lambda_{\max}^{\text{block}}}{\tau \lambda_{\min}^{\text{nz}}}\log \tfrac{1}{\varepsilon} \right), \]
where $\lambda_{\min}^{\text{nz}}$ denotes the smallest non-zero eigenvalue of $A A^T$.  Thus, a convergence rate depending on the  geometric properties of the matrix $A$ and its submatrices $A_J$ and on the size of the  blocks $\tau=|J|$. When comparing RBK with basic Kaczmarz in terms of  total computational cost to achieve an $\varepsilon$ solution we get:
\[     {\cal O} \left( \tau n \cdot  \tfrac{m \lambda_{\max}^{\text{block}}}{\tau \lambda_{\min}^{\text{nz}}}\log \tfrac{1}{\varepsilon} \right)   \quad \text{vrs.} \quad  {\cal O} \left( n \cdot \tfrac{ m }{\lambda_{\min}^{\text{nz}}}\log \tfrac{1}{\varepsilon} \right).    \] 
Therefore, our convergence rate also explains \textit{why} and \textit{when} the randomized block Kaczmarz algorithm with the constant extrapolated stepsize \eqref{eq:step_c} or adaptive extrapolated stepsize \eqref{eq:extrapol_alp}   works better compared to its basic counterpart.  In particular, the analysis reveals that a distributed implementation of  extrapolated  RBK algorithm  is most effective when the sampling of the blocks of rows  yields  a partition  into well-conditioned blocks, that is,  the stochastic conditioning parameter $\lambda_{\max}^{\text{block}}$ is small.

\noindent For the third choice of the extrapolated stepsize, depending on the roots of Chebyshev polynomials, and for normalized matrices having $\lambda_{\min} >0$ we get a linear rate for the expected iterates of the form (see  Table~\ref{tbl:complexity}):
\[ {\cal O} \left( \sqrt{\tfrac{\lambda_{\max}}{\lambda_{\min}}} \log \tfrac{1}{\varepsilon} \right), \]
where $\lambda_{\min} (\lambda_{\max})$  denote the smallest (largest) eigenvalue of  $A A^T$, respectively. Note that this convergence estimate is the same as for the conjugate gradient method and it is optimal for this class of iterative schemes, as the condition number of the matrix is square rooted. 
%Moreover, since this rate does not depend on the size of the blocks $|J|$, then we usually implement this accelerated variant by sampling single rows $|J|=1$.  

%%%%%%%%%%%%%%%%%%%%%%%%%%%%%%%%%%%%%%%%
%%%%%%%%%%%%%%%%%%%%%%%%%%%%%%%%%%%%%%%%

\section{Preliminaries}
\label{sec:prelim}
Note that the problem of finding a solution of the linear system $Ax=b$ can be posed as a quadratic optimization problem, the so-called linear least-square problem:
\begin{align}
\label{eq:ls}
\min_{x \in \rset^n} \frac{1}{2m}\|Ax - b\|^2 \qquad \left(:= \frac{1}{2m}\sum_{i=1}^m (a_i^Tx-b_i)^2 \right).
\end{align}
A more particular formulation is to  find the least-norm solution of the linear system:
\begin{align}
\label{eq:lsnorm}
\min_{x \in \rset^n} \frac{1}{2} \|x\|^2 \quad \text{s.t.} \quad Ax = b.
\end{align}
The dual of optimization problem \eqref{eq:lsnorm} takes also the form of a quadratic program:
\begin{align}
\label{eq:lsdual}
\min_{y \in \rset^m} \frac{1}{2}\|A^Ty\|^2 -b^Ty,
\end{align}
where the primal variable $x$ and the dual variable $y$ are related through the relation   $x=A^Ty$. Let us define the primal and dual objective functions $f(x)=(1/2m) \|Ax - b\|^2$ and  $g(y) = 1/2 \|A^Ty\|^2 -b^Ty$, respectively. Recall that the set of solutions is denoted $\cX=\{x \in \rset^n: \; Ax=b\}$ and for any given $x$ we define its projection onto $\cX$ by  $x^*= \Pi_{\cX}(x)$.

%%%%%%%%%%%%%%%%%%%%%%%%%%%%%%%%%%%%%%%%

\subsection{Basic  Kaczmarz algorithm}
The   Kaczmarz  algorithm is an iterative  scheme for solving the linear system $Ax=b$ that requires only ${\cal O}(n)$ cost per iteration and storage and has a linear  rate of convergence.  At each iteration $k$, the algorithm  selects (cyclically, randomly) a row $i_k \in [m]$ of the linear system  and does an orthogonal projection of the current estimate vector $x^k$ onto the corresponding hyperplane $a_{i_k}^T x = b_{i_k}$:
\[  \min_x  \|x - x^k\|^2 \quad \text{s.t.} \quad  a_{i_k}^T x = b_{i_k}.   \]
Then, we choose the next iterate along the line connecting the current iterate and the projection. This leads to the following iteration for randomized/cyclic Kaczmarz algorithm \cite{Kac:37,StrVer:09}:
\begin{algorithm}[ht!]
\caption{(Algorithm Kaczmarz)}
\label{alg:K}
\begin{algorithmic}[1]
\STATE choose $x^0 \in \rset^n$
\FOR { $k \geq 0$ }
\STATE choose an index  $ i_k \in [m]$ (random, cyclic) and  update:
\STATE  $x^{k+1} = x^k - \alpha_k \frac{a_{i_k}^T x^k - b_{i_k}}{\|a_{i_k}\|^2} a_{i_k}$.
\ENDFOR
\end{algorithmic}
\end{algorithm}

\noindent Usually, $\alpha_k$ is chosen constant in  interval $(0, \; 2)$. For $\alpha_k=1$ we recover  basic Kaczmarz algorithm \cite{Kac:37}.

%%%%%%%%%%%%%%%%%%%%%%%%%%

\subsection{Interpretations}  We can view randomized Kaczmarz algorithm, i.e. when $i_k$ is chosen randomly, as an optimization method for solving a specific primal or dual optimization problem. More  precisely, Kaczmarz algorithm is a particular case of:\\

\noindent \textit{SGD (Stochastic Gradient Descent)}: The randomized Kaczmarz (Algorithm  \ref{alg:K}) is equivalent to one step of the stochastic gradient descent method \cite{NemJud:09} applied to the  finite sum problem \eqref{eq:ls}. Specifically, a component function $i_k$, $f_{i_k}(x) = 1/2 (a_{i_k}^T x - b_{i_k})^2$, is chosen randomly and a negative gradient step (having  $\nabla f_{i_k}(x) = (a_{i_k}^T x - b_{i_k}) a_{i_k}$) of this partial function in $x^k$ with stepsize $\alpha_k/\|a_{i_k}\|^2$ is considered:
\[  x^{k+1} = x^k - \frac{\alpha_k}{\|a_{i_k}\|^2} \nabla f_{i_k}(x^k).  \]

\noindent \textit{RCD (Random Coordinate Descent)}: The randomized Kaczmarz (Algorithm  \ref{alg:K})    is equivalent to one step of randomized coordinate descent method \cite{Nes:12} applied to the dual problem \eqref{eq:lsdual}. Specifically, a negative gradient step in the random $i_k$th component of $y$ (having the expression $\nabla_{i_k} g(y) = a_{i_k}^T A^T y - b_{i_k}$) with stepsize $\alpha_k/\|a_{i_k}\|^2$ is taken, yielding:
\[  y^{k+1} = y^k - \frac{\alpha_k}{\|a_{i_k}\|^2} \nabla_{i_k} g(y^k) \cdot e_{i_k}, \]
where $e_i$ denotes the $i$th column of the identity matrix in $\rset^{n \times n}$.  We recover easily  the iteration of Algorithm  \ref{alg:K}   by simply multiplying this update with $A^T$ and   using the relation between the primal and dual variables given by  $x=A^Ty$. Note that in both interpretations, we need to choose a specific stepsize, in order to prove convergence, see \cite{NemJud:09,Nes:12}.

%%%%%%%%%%%%%%%%%%%%%%%%%%%%%%%%

\subsection{Convergence properties}
\label{sec:convbasicK}
It is known that Algorithm  \ref{alg:K} converges to the minimum norm solution of $Ax = b$ when it is initialized with $x^0 = 0$, but the speed of convergence is not simple to quantify, and especially, depends on the ordering of the rows \cite{DeuHun:97}. The situation changes if one considers a randomization such that in each step one chooses a row of the system matrix at random, according to a probability $\textbf{P}$. In the seminal paper \cite{StrVer:09} it has been shown that sampling the rows of $A$ with probability $\textbf{P} (i =i_k) = \frac{\|a_{i_k}\|^2}{\|A\|_F^2}$  for all $i \in [m]$  and using constant stepsize $\alpha=1$, we get a linear convergence rate in expectation of the form:
\begin{align}
\label{eq:ratebasic}
\Exp{\|x^k-x^*_k\|^2} \leq \left( 1  - \frac{\lambda_{\min}^{\text{nz}}(A^TA)}{\|A\|_F^2} \right)^k \|x^0 - x^*_0\|^2,
\end{align}
where $\lambda_{\min}^{\text{nz}}(\cdot)$ denotes the minimum non-zero eigenvalue of a given  matrix and $x^*_k = \Pi_{\cX}(x^k)$. For completeness, let us derive this convergence rate.
Considering the stepsize $\alpha_k$  constant in the interval $(0, \; 2)$ and using that  $\langle x-x^*, (a_i^Tx-b_i) a_i \rangle = (a_i^Tx-b_i)^2$ for any $x^*$ a solution of $Ax=b$,  we get:
\begin{align*}
\|x^{k+1} - x^{*}\|^2 & = \| x^k - x^{*}\|^2 - 2 \alpha   \frac{(a_i^T x^k - b_i)^2}{\|a_i\|^2}  + \alpha^2  \frac{(a_i^T x^k - b_i)^2}{\|a_i\|^2} \\
&=  \| x^k - x^{*}\|^2 - \alpha(2 - \alpha) \frac{(a_i^T x^k - b_i)^2}{\|a_i\|^2}.
\end{align*}
Taking now the conditional expectation under the  probability $\textbf{P} (i =i_k) = \frac{\|a_{i_k}\|^2}{\|A\|_F^2}$, we get:
\begin{align*}
\Expi{\|x^{k+1} - x^{*}\|^2 | x^k} & \leq  \| x^k - x^{*}\|^2 -  \frac{\alpha (2- \alpha)}{\|A\|_F^2} \|A x^k - b\|^2.
\end{align*}
Further, it is well known from the Courant-Fischer theorem  that for any matrix $A$ we have $\|A x \|^2 \geq \lambda_{\min}^{\text{nz}}(AA^T) \|x\|^2$ for all $x \in \text{range}(A^T)$. Moreover, we have  that $x - \Pi_{\cX}(x) \in \text{range}(A^T)$ for any $x$. In conclusion, if we denote $x_k^* = \Pi_{\cX}(x^k) $, we get:
\[  \|A x^k - b\|^2 = \|A (x^k - x^{*}_k) \|^2 \geq \lambda_{\min}^{\text{nz}}(AA^T) \|x^k - x^{*}_k\|^2.  \]
Using this inequality in the recurrence above and taking expectation over the entire history we get the following linear convergence rate in expectation:
\begin{align}
\label{eq:crbasic}
\Exp{\|x^{k+1} - x^{*}_k\|^2}  \leq  \left(1 - \frac{ \alpha (2- \alpha) \lambda_{\min}^{\text{nz}}(AA^T)}{\|A\|_F^2} \right) \Exp{\| x^k - x^{*}_k\|^2}.
\end{align}
\noindent For the optimal choice $\alpha^*=1$ (i.e.   $\alpha^* = \arg \max_\alpha  \alpha (2- \alpha)$) we get the simpler convergence estimate  \eqref{eq:ratebasic}  derived in \cite{StrVer:09}. Note that for  ill-conditioned problems, i.e. $\lambda_{\min}^{\text{nz}}(AA^T)$  small and $\|A\|_F$ large, this linear convergence is very slow using a constant stepsize $\alpha \in (0, \; 2)$. In the next sections we prove that block variants of  randomized Kaczmarz (Algorithm  \ref{alg:K}) with properly chosen extrapolated  stepsize $\alpha_k$ larger than $2$ can substantially  accelerate the convergence rate \eqref{eq:crbasic}.

%%%%%%%%%%%%%%%%%%%%%%%%%%%%%%%%%%%

\subsection{Preliminary probability results}
\label{sec:prob}
Let $J$ be a random set-valued map with values in $2^{[m]}$. Any realization  $J \subseteq [m]$ of this random variable, referred to as sampling and having the same notation as the random variable, is characterized by the probability distribution $\textbf{P}(J)$. We also define   the probability with which an index $i \in [m]$ can be found in $J$ as:
\[ p_i = \textbf{P}(i \in J). \]
Then, for any  scalars  $\theta_i$, with $i \in [m]$, the following relation holds in expectation:
\begin{align}
\label{eq:pr_prob}
\ExpJ{\sum_{i \in J} \theta_i} =  \sum_{J \subseteq [m]}  \left(\sum_{i \in J} \theta_i \right)  \textbf{P}(J) = \sum_{i \in [m]}  \theta_i  \left(\sum_{J: i \in J} \textbf{P}(J)\right)  = \sum_{i \in [m]} p_i \theta_i.
\end{align}
The following examples for  sampling blocks of rows of $A \in \rset^{m \times n}$ will be used in our subsequent analysis.

\noindent \textit{Uniform sampling}:  One natural choice is the  \textit{uniform} sampling of $\tau$ unique indexes of rows that make up $J$, i.e. $|J| = \tau$ for all samplings, with $1 \leq \tau \leq m$ fixed. For this choice of the random variable $J$, we observe that  we have a total number of  $\binom{m}{\tau}$ possible values that $J$ can take. Thus, for the uniform sampling we have  $\textbf{P}(J) = 1/\binom{m}{\tau}$.  We can also express $p_i$ for the uniform sampling as:
\begin{align}
\label{eq:piunif}
p_i =  \textbf{P}(i \in J) =   \sum_{J: i \in J}  \textbf{P}(J) =   \frac{\binom{m-1}{\tau-1}}{\binom{m}{\tau}} = \frac{\tau}{m}.
\end{align}

\noindent \textit{Partition sampling}: Another choice is the \textit{partition} sampling, i.e. consider a partition of $[m]$ given by $\{J_1, \cdots, J_\ell \}$, and then take $\textbf{P}(J) = 1/\ell$ or  $\textbf{P}(J) = \|A_J\|_F^2/\|A\|_F^2$ for all $J  \in \{J_1, \cdots, J_\ell \}$.  For example, for the first probability choice of  the  partition sampling,  $p_i$  is given by:
\begin{align}
\label{eq:pipart}
p_i =  \frac{1}{\ell}.
\end{align}
In particular,  if all the subsets in the partition have the same cardinality, i.e. $|J_l| =\tau$ for all $l \in [\ell]$, and $A$ is normalized, then the two probabilities are the same and  $\ell = m/\tau$. Hence,  $p_i = \frac{\tau}{m}$. These preliminary results will help us in the convergence analysis of   randomized block Kaczmarz algorithms we propose next.

%%%%%%%%%%%%%%%%%%%%%%%%%%%%%%%%%%%%%
%%%%%%%%%%%%%%%%%%%%%%%%%%%%%%%%%%%

\section{Randomized  block Kaczmarz algorithms}
\label{sec:RBK}
In this section we design new variants  of randomized Kaczmarz, Algorithm  \eqref{alg:K},  considering at each step a block  of rows of the linear system $Ax=b$ and different choices for the  stepsize. For all these methods we prove expected linear convergence rates. Note that block Kaczmarz methods have been also considered in other works, see e.g. \cite{BauCom:06,Cen:81,Mer:63,Pie:84} and the references therein. Nevertheless, to our knowledge, this paper is the first one that provides an expected linear rate of convergence that depends explicitly on geometric properties of the system matrix $A$ and its submatrices $A_J$.  Moreover, the convergence estimates hold for several extrapolated stepsizes.   In our Randomized Block Kaczmarz (RBK) algorithm,  at each iteration, instead of projecting on only one hyperplane, we consider projections onto several hyperplanes and then take as a new direction a  convex combination of these projections with some stepsize (see Algorithm \ref{alg:RBK}).
\begin{algorithm}[h!]
\caption{(Algorithm RBK)}
\label{alg:RBK}
\begin{algorithmic}[1]
\STATE choose $x^0 \in \rset^n$,  stepsize sequence $(\alpha_k)_{k \geq 0}$,  and weights sequence  $(\omega_k)_{k \geq 0}$
\FOR { $k \geq 0$ }
\STATE draw  sample  $J_k \sim \textbf{P}$ and  update:
\STATE  $x^{k+1} = x^k -  	\alpha_k  \left(\sum\limits_{i \in J_k} \omega_{k}^i \frac{a_i^T x^k - b_i}{\|a_i\|^2} a_i \right)$.
\ENDFOR
\end{algorithmic}
\end{algorithm}
Here $J_k = \{ i_k^1,\cdots,i_k^{\tau_k} \} \subseteq [m]$ is the set of indexes corresponding to the rows selected at iteration $k$ of size $\tau_k \in [1, m]$ and \textbf{P} denotes the probability distribution over the collection of subsets of indexes of $[m]$. Moreover, the weights $ \omega_k =(\omega_{k}^{i})_{i \in J_k}$ are chosen positive and summing to 1. Thus,   in our analysis we assume bounded weights satisfying  $0 < \omega_{\min} \leq  \omega_k^i  \leq \omega_{\max} < 1$ for all $i \in J_k$ and $k \geq 0$.  Two simple choices for  the weights are e.g. $\omega_{k}^i= \|a_i\|^2/\sum_{i \in J_k} \|a_i\|^2$ or $\omega_k^i = 1/\tau_k$ for all  $k \geq 0$. In these two particular cases we get the following compact updates:
\[  x^{k+1} = x^k  - \alpha_k \frac{A_{J_k}^T (A_{J_k} x^k - b_{J_k})}{\|A_{J_k}\|_F^2} \quad \text{or} \quad x^{k+1} = x^k  - \alpha_k \frac{A_{J_k}^T D_{J_k} (A_{J_k} x^k - b_{J_k})}{\tau_k},  \]
respectively, where the diagonal matrix  $D_{J} = \text{diag}(1/\|a_i\|^{2}, \; i \in J) \in \rset^{\tau \times \tau}$. Several choices for the stepsize  will be given in the next sections, based on over-relaxations (extrapolations), i.e. $\alpha_k >2$.  Similarly, as for  Kaczmarz  algorithm, RBK (Algorithm  \ref{alg:RBK}) can be interpreted as:\\

\noindent \textit{BSGD (Batch Stochastic Gradient Descent)}:  One iteration of RBK algorithm  can be viewed as one step of the batch stochastic gradient descent \cite{NemJud:09} applied to the  finite sum problem \eqref{eq:ls} when the weights $\omega_k$ are chosen in a particular fashion. Specifically, if we choose the particular weights $\omega_{k}^i = \|a_i\|^2 /\sum_{i \in J_k} \|a_i\|^2$ and uniform probability, then we recover  the batch stochastic gradient descent method with a certain choice of the stepsize:
   \[ x^{k+1} = x^k -  	\frac{\tau_k  \alpha_k}{\sum_{i \in J_k} \|a_i\|^2}  \left( \frac{1}{\tau_k} \sum\limits_{i \in J_k} (a_i^T x^k - b_i)a_i \right).  \]

\noindent \textit{RBCD (Randomized Block Coordinate Descent)}:  One iteration of RBK algorithm  can be viewed as one step of the block coordinate descent method \cite{NecCli:16,Nes:12} applied to the dual problem \eqref{eq:lsdual}  when the weights $\omega_k$ are chosen in a particular fashion. Specifically, if we choose the particular weights $\omega_{k}^i = \|a_i\|^2 /\sum_{i \in J_k} \|a_i\|^2$, then we recover  the block coordinate descent  method with a certain choice of the stepsize:
\[ x^{k+1} = x^k -  	\frac{\alpha_k}{\sum_{i \in J_k} \|a_i\|^2}  \left( \sum\limits_{i \in J_k} (a_i^T x^k - b_i)a_i \right).  \]
However, for general weights $\omega_k$ and stepsize $\alpha_k$, RBK  algorithm  cannot be interpreted in these ways, thus our scheme is more general.  In the following, we  denote $x_{k}^* = \Pi_{\cX} (x^k)$, that is the projection of $x^k$ onto the solution set $\cX$ of the linear system $Ax=b$.

%%%%%%%%%%%%%%%%%%%%%%%%%%%%%%%%%%%%%%%%%%%

\subsection{Randomized block Kaczmarz algorithm with constant stepsize}
\noindent In this section we investigate the convergence rate of RBK algorithm  for constant extrapolated stepsize $\alpha_k = \alpha >2$ and weights $\omega_k^i = \omega_i$  for all $k$. Thus, the iteration of RBK (Algorithm \ref{alg:RBK}) becomes in this case:
\begin{align}
\label{eq:brk_c}
x^{k+1} = x^k - \alpha  \left(\sum\limits_{i \in J_k} \omega_i \frac{a_i^T x^k - b_i}{\|a_i\|^2} a_i \right).
\end{align}
The weights are chosen  to satisfy $0 < \omega_{\min} \leq \omega_i \leq \omega_{\max} < 1$ for all $i$ and sum to $1$. Let us also define the following stochastic conditioning  parameter depending on the geometric properties of the submatrices $A_J$:
\[  \lambda_{\max}^{\text{block}} = \max_{J \sim \textbf{P}} \lambda_{\max} \left( A_J^T \diag \left( \frac{1}{\|a_i\|^2},i \in J\right) A_J \right). \]
Then, we consider an extrapolated  constant stepsize of the form:
\begin{align}
\label{eq:step_c}
0 <  \alpha < \frac{2 \omega_{\min}}{\omega_{\max}^2 \lambda_{\max}^{\text{block}}}.
\end{align}
When we choose a random variable such that  all the samplings satisfy $|J| = \tau$, with $\tau \in [1, m]$,  then it is straightforward to see that $\lambda_{\max}^{\text{block}} < \tau$ provided that $\text{rank}(A_J) \geq 2$. Hence,  in this case we use an over-relaxed (extrapolated) stepsize, since usually $ 2 \omega_{\min}/\omega_{\max}^2 \lambda_{\max}^{\text{block}} > 2$. For example, for $\omega_i=1/\tau$, we get  $2 \tau/\lambda_{\max}^{\text{block}} >2$.  Using \eqref{eq:pr_prob} we also define the positive semidefinite  matrix $W$ as:
\[ W = \ExpJ{ \sum\limits_{i \in J} \frac{a_i a_i^T}{\|a_i\|^2}} = \sum_{i \in [m]} p_i  \frac{a_i a_i^T}{\|a_i\|^2} = A^T \diag \left( \frac{p_i}{\|a_i\|^2},i \in [m]\right) A. \]
From our best knowledge, the choice  \eqref{eq:step_c} for the stepsize  in the block Kaczmarz algorithm seems to be new. The next theorem proves the convergence rate of this algorithm which depends explicitly on the  geometric properties of the system matrix $A$ and its submatrices $A_J$.

\begin{theorem}
\label{th1:convergence}
Let $\{ x^k \}_{k \ge 0}$ be generated by  RBK (Algorithm \ref{alg:RBK}) with the particular  update \eqref{eq:brk_c}, i.e. the weights  satisfy $0 < \omega_{\min} \leq \omega_i \leq \omega_{\max} < 1$ for all $i \in [m]$ and the  stepsize $ \alpha=\frac{(2 - \delta) \omega_{\min}}{\omega_{\max}^2 \lambda_{\max}^{\text{block}}}$  for some $\delta \in (0, 1]$. Then, we have the following linear convergence rate in expectation:
\begin{align}
\label{eq:convrate1}
 \Exp{\|x^{k} - x_{k}^*\|^2}  \leq  \left(  1 -  \frac{(2-\delta)\omega_{\min}^2  \lambda_{\min}^{\text{nz}}(W)}{\omega_{\max}^2 \lambda_{\max}^{\text{block}}} \right)^k  \| x^0 - x_{0}^*\|^2.
\end{align}
\end{theorem}

\begin{proof}
Since we assume a consistent linear system, that is there is $x^*$ such that $A x^* = b$, we have:
\begin{align*}
\|x^{k+1} - x^*\|^2 & = \| x^k  - x^*  - \alpha \left(\sum\limits_{i \in J_k} \omega_i \frac{a_i^T x^k - b_i}{\|a_i\|^2} a_i \right) \|^2\\
&= \| x^k  - x^*  - \alpha \left(\sum\limits_{i \in J_k} \omega_i \frac{a_i a_i^T}{\|a_i\|^2} (x^k - x^*) \right) \|^2\\
& = \|  \left( I_n -  \alpha \left(\sum\limits_{i \in J_k} \omega_i \frac{a_i a_i^T}{\|a_i\|^2} \right)\right) (x^k-x^*) \|^2\\
&= (x^k - x^*)^T \left( I_n - 2 \alpha \sum\limits_{i \in J_k} \omega_i \frac{a_i a_i^T}{\|a_i\|^2} + \alpha^2 \left(\sum\limits_{i \in J_k} \omega_i \frac{a_i a_i^T}{\|a_i\|^2} \right)^2 \right)(x^k - x^*).
\end{align*}
We need to take conditional expectation over  $J_k$.   However, for a general random sampling $J$ we have from \eqref{eq:pr_prob} that the expectation over the first sum from above  yields the lower bound:
\begin{align*}
\ExpJ{ \sum\limits_{i \in J} \omega_i \frac{a_i a_i^T}{\|a_i\|^2}} & \succeq (\min_{i \in J} \omega_i) \ExpJ{\sum\limits_{i \in J}  \frac{a_i a_i^T}{\|a_i\|^2}} = \omega_{\min}  \sum_{i \in [m]} p_i  \frac{a_i a_i^T}{\|a_i\|^2} \\
& =  \omega_{\min}  A^T \diag \left( \frac{p_i}{\|a_i\|^2},i \in [m]\right) A =   \omega_{\min} W.
\end{align*}
Thus, we obtained:
\begin{align}
\label{eq:exp1}
\ExpJ{ \sum\limits_{i \in J} \omega_i \frac{a_i a_i^T}{\|a_i\|^2}} \succeq \omega_{\min}  W.
\end{align}
Moreover,  using that for any $Q \succeq 0$ we have $Q^2 \preceq \lambda_\text{max}(Q) Q$, the expectation over the second sum also yields the following upper bound:
\begin{align*}
\ExpJ{\left(\sum\limits_{i \in J} \omega_i \frac{a_i a_i^T}{\|a_i\|^2} \right)^2} & \preceq \ExpJ{ \lambda_{\max} \left(\sum\limits_{i \in J} \omega_i \frac{a_i a_i^T}{\|a_i\|^2} \right) \left(\sum\limits_{i \in J} \omega_i \frac{a_i a_i^T}{\|a_i\|^2} \right)}\\
& \preceq (\max_{i \in J} \omega_i) \ExpJ{ \lambda_{\max} \left(\sum\limits_{i \in J}  \frac{a_i a_i^T}{\|a_i\|^2} \right) \left(\sum\limits_{i \in J} \omega_i \frac{a_i a_i^T}{\|a_i\|^2} \right)}\\
& \preceq (\max_{i \in J} \omega_i) \ExpJ{ \lambda_{\max} \left( A_J^T \diag \left( \frac{1}{\|a_i\|^2},i \in J\right) A_J \right) \left(\sum\limits_{i \in J} \omega_i \frac{a_i a_i^T}{\|a_i\|^2} \right)}\\
& \preceq (\max_{i \in J} \omega_i)^2 \lambda_{\max}^{\text{block}}  \ExpJ{ \sum\limits_{i \in J}  \frac{a_i a_i^T}{\|a_i\|^2} } = \omega_{\max}^2 \lambda_{\max}^{\text{block}}  W,
\end{align*}
where recall that  $\lambda_{\max}^{\text{block}} =\max_{J \sim \textbf{P}} \lambda_{\max} \left( A_J^T \diag \left( \frac{1}{\|a_i\|^2},i \in J\right) A_J \right)$. Therefore,  taking conditional expectation w.r.t. the block $J_k$ over entire history ${\cal F}_k =\{J_0, \cdots,J_{k-1}\}$ in the  recurrence above, we get:
\begin{align*}
\ExpJ{\|x^{k+1} - x^*\|^2|{\cal F}_k} \leq (x^k-x^*)^T \left( I_n - 2\alpha \omega_{\min}  W + \alpha^2 \omega_{\max}^2 \lambda_{\max}^{\text{block}}  W \right) (x^k-x^*).
\end{align*}
In order to ensure decrease we need $2\alpha \omega_{\min}  - \alpha^2 \omega_{\max}^2 \lambda_{\max}^{\text{block}} \geq 0$, that is we get an extrapolated stepsize:
\[  \alpha \leq \frac{2 \omega_{\min}}{\omega_{\max}^2 \lambda_{\max}^{\text{block}}}, \]
and the optimal stepsize is obtained by maximizing $2\alpha \omega_{\min}  - \alpha^2 \omega_{\max}^2 \lambda_{\max}^{\text{block}}$ in $\alpha$ which leads to:
\[  \alpha^* = \frac{\omega_{\min}}{\omega_{\max}^2 \lambda_{\max}^{\text{block}}}. \]
Hence, taking  stepsize $\alpha = (2-\delta) \omega_{\min}/\omega_{\max}^2 \lambda_{\max}^{\text{block}}$ for some $\delta \in (0, 1]$, we get:
\begin{align*}
\ExpJ{\|x^{k+1} - x^*\|^2|{\cal F}_k} \leq (x^k-x^*)^T \left( I_n - (2-\delta) \frac{\omega_{\min}^2 }{\omega_{\max}^2 \lambda_{\max}^{\text{block}}}  W \right) (x^k-x^*).
\end{align*}
On the other hand, it is well-known from the Courant-Fischer theorem  that for any matrix $A$ we have $\|A x \|^2 \geq \lambda_{\min}^{\text{nz}}(AA^T) \|x\|^2$ for all $x \in \text{range}(A^T)$. Moreover, we have  that $x - \Pi_{\cX}(x) \in \text{range}(A^T)$ for any $x$. In conclusion, using that  $W = A^T D A$ with the diagonal matrix $D = \diag \left( \frac{p_i}{\|a_i\|^2},i \in [m]\right)$ invertible, we get that:
\begin{align*}
(x^k - x_{k}^*)^T W (x^k - x_{k}^*) & =  \|D^{1/2} A (x^k - x_{k}^*) \|^2 \geq \lambda_{\min}^{\text{nz}}(A^T D A) \|x^k - x_{k}^*\|^2\\
& =\lambda_{\min}^{\text{nz}}(W) \|x^k - x_{k}^*\|^2.
\end{align*}
Using this inequality in the recurrence above and taking expectation over the entire history we get:
\[  \Exp{\|x^{k+1} - x_{k+1}^*\|^2}  \leq  \left(  1 - (2-\delta) \frac{\omega_{\min}^2  \lambda_{\min}^{\text{nz}}(W)}{\omega_{\max}^2 \lambda_{\max}^{\text{block}}}  \right) \Exp{\| x^k - x_{k}^*\|^2},  \]
which shows an expected  linear convergence rate for RBK depending on the  parameters   $\lambda_{\min}^{\text{nz}}(W)$ and $\lambda_{\max}^{\text{block}}$ associated to the system  matrix $A$ and its submatrices $A_J$, respectively.
\end{proof}

\noindent Now, let us consider the uniform and partition sampling examples of Section \ref{sec:prob} where all the blocks  sampling have the same size  $|J| = \tau$. In this case  we have $p_i = \frac{\tau}{m}$. Let us also consider the particular choices  $\delta=1$, weights $\omega_i=1/\tau$, and  matrices $A$ with normalized rows, i.e. $\|a_i\| =1$ for all $i \in [m]$. Hence, $\|A\|_F^2=m$.   Then,  our convergence rate \eqref{eq:convrate1} becomes:
\begin{align}
\label{eq:converate11}
\Exp{\|x^{k} - x_{k}^*\|^2}  \leq  \left(  1 - \frac{\tau}{\lambda_{\max}^{\text{block}}} \frac{\lambda_{\min}^{\text{nz}}(A^TA)}{m}  \right)^k \| x^0 - x_{0}^*\|^2.
\end{align}
Comparing with the convergence rate \eqref{eq:ratebasic} of the basic Kaczmarz method (recall that for normalized matrices $\|A\|_F^2=m$) we get an improvement  $\frac{\tau}{\lambda_{\max}^{\text{block}}} >1, $ which shows that  for RBK algorithm  with the new extrapolated stepsize \eqref{eq:step_c}  we can get a speed-up even of order approximately $\tau$ compared to basic Kaczmarz algorithm on matrices with well-conditioned blocks (i.e. on matrices having $\lambda_{\max}^{\text{block}} \ll \tau$). Section  \ref{sec:pave} provides  choices for the sampling that lead to a small stochastic conditioning parameter  $\lambda_{\max}^{\text{block}}$.

%%%%%%%%%%%%%%%%%%%%%%%%%%%%%%%%%%%%

\subsection{Randomized block Kaczmarz algorithm with adaptive stepsize}
Since the previous algorithm involves a stepsize depending on $\lambda_{\max}^{\text{block}}$, which may be difficult to compute in large-scale settings (i.e. when the random variable $J$ is complicated and the number of rows $m$ is large), in this section we design a randomized block  Kaczmarz  algorithm with an adaptive stepsize, which doe not require the computation of $\lambda_{\max}^{\text{block}}$. More precisely,  we consider a variant of  RBK (Algorithm \ref{alg:RBK}) with variable weights and  an adaptive stepsize  approximating  online $\lambda_{\max}^{\text{block}}$.  For simplicity of the notation let us define $\bar{\omega}_i^{k} = \frac{\omega_i^{k}}{\|a_i\|^2}$. Then, we consider the iteration of RBK (Algorithm \ref{alg:RBK})  with an adaptive extrapolated stepsize of the form:
\begin{align}
\label{adaptstep}
0 < \alpha_k < 2 L_k, \quad  \text{where} \quad   L_k = \begin{cases}
\frac{\sum_{i\in J_k} \bar{\omega}_{i}^{k} (a_i^T x^k - b_i)^2 }{\norm{ \sum_{i\in J_k} \bar{\omega}_{i}^{k} (a_i^T x^k - b_i)a_i}^2} \quad \text{if} \quad a_i^T x^k - b_i \not=0 \;\; \forall i \in J_k\\
\frac{1}{\lambda_{\max}\left( A_{J_k}^T \diag \left( \bar{\omega}_i^{k},i \in J_k \right) A_{J_k} \right) }  \quad \text{otherwise.}
\end{cases}
\end{align}

\noindent Note that we do not need to compute $L_k$ for the second case when implementing the algorithm.  Recall that we consider weights satisfying $0 < \omega_{\min} \leq \omega_i^k \leq \omega_{\max} < 1$ for all $k, i$, and summing to $1$. Hence, from Jensen's inequality we always have $L_k \geq 1$ and consequently $2 L_k \geq 2$, i.e. we use extrapolation.  Further, in our convergence analysis we take  a stepsize of the form $\alpha_k = (2-\delta) L_k$ for some $\delta \in (0, \; 1]$. Moreover, we denote $x_{k}^* = \Pi_{\cX} (x^k)$, that is the projection of $x^k$ onto the solution set of the linear system. It has been observed in practice that block Kaczmarz iteration with this adaptive  choice for the stepsize has better performances  than the same  algorithm but with stepsize $\alpha_k \in (0, 2)$, see e.g.    \cite{BauCom:06,Cen:81,CenChe:12,Mer:63,Pie:84}. However, from our knowledge, there is no theory explaining \textit{why} and \textit{when} this adaptive method works. The next theorem proves the convergence rate of the adaptive algorithm depending  explicitly on the  geometric properties of the system matrix $A$ and its submatrices $A_J$ and answers  to the question related to the theoretical understanding of observed practical efficiency of extrapolated block Kaczmarz methods.

\begin{theorem}
\label{th2:convergence}
Let $\{ x^k \}_{k \ge 0}$ be generated by RBK  (Algorithm \ref{alg:RBK}) with the adaptive stepsize $\alpha_k = (2-\delta) L_k$ for some $\delta \in (0, 1]$ and the  weights satisfying $0 < \omega_{\min} \leq \omega_i^k \leq \omega_{\max} < 1$ for all $k, i$. Then, we have the following linear convergence in expectation:
\begin{align}
\label{eq:convrate2}
\Exp{\|x^{k} - x_{k}^{*} \|^2} & \leq \left(1  -   \frac{\delta (2- \delta) \omega_{\min} \lambda_{\min}^{\text{nz}}(W)}{\omega_{\max} \lambda_{\max}^{\text{block}}} \right)^k \| x^0 - x_{0}^{*}\|^2.
\end{align}
\end{theorem}

\begin{proof}
Using that  $\langle x-x^*, (a_i^Tx-b_i) a_i \rangle = (a_i^Tx-b_i)^2$ in the update of RBK, we get:
\begin{eqnarray*}
&& \|x^{k+1} - x_{k+1}^*\|^2 =\left \|   x^k -  	\alpha_k  \left(\sum\limits_{i \in J_k} \omega_{k}^i \frac{a_i^T x^k - b_i}{\|a_i\|^2} a_i \right) - x_{k+1}^*\right \|^2\\
&& = \| x^k - x_{k+1}^*\|^2 - 2 \alpha_k \left(  \sum\limits_{i \in J_k} \omega_{i}^{k}  \frac{(a_i^T x^k - b_i)^2}{\|a_i\|^2} \right) + \alpha_k^2 \left\|  \sum\limits_{i \in J_k} \omega_{i}^{k}  \frac{(a_i^T x^k - b_i)^2}{\|a_i\|^2} a_i \right\|^2\\
&&=  \| x^k - x^{k+1,*}\|^2 - 2(2 - \delta) \frac{\left(\sum\limits_{i \in J_k} \bar{\omega}_{i}^{k} (a_i^T x^k - b_i)^2\right)^2}{\|  \sum\limits_{i \in J_k} \bar{\omega}_{i}^{k} (a_i^T x^k - b_i) a_i\|^2} + (2-\delta)^2 \frac{\left(\sum\limits_{i \in J_k} \bar{\omega}_{i}^{k} (a_i^T x^k - b_i)^2\right)^2}{\|  \sum\limits_{i \in J_k} \bar{\omega}_{i}^{k} (a_i^T x^k - b_i) a_i\|^2} \\
&&= \| x^k - x^{k+1,*}\|^2 - \delta (2 - \delta) \frac{\sum\limits_{i \in J_k} \bar{\omega}_{i}^{k} (a_i^T x^k - b_i)^2}{\| \sum\limits_{i \in J_k} \bar{\omega}_{i}^{k} (a_i^T x^k - b_i) a_i\|^2} \sum\limits_{i \in J_k} \bar{\omega}_{i}^{k} (a_i^T x^k - b_i)^2\\
&&= \| x^k - x^{k+1,*}\|^2 - \delta (2- \delta) L_k \sum\limits_{i \in J_k} \bar{\omega}_{i}^{k} (a_i^T x^k - b_i)^2.
\end{eqnarray*}
Note that we get the same recurrence also for the trivial choice $L_k=1/\lambda_{\max}\left( A_{J_k}^T \diag \left( \bar{\omega}_i^{k},i \in J_k \right) A_{J_k} \right)$. Now let us bound $L_k$. For the nontrivial case, using that $\lambda_{\max} (MN) = \lambda_{\max} (NM)$ for any two matrices $M$ and $N$ of appropriate dimensions, we have:
\begin{align*}
L_k & = \frac{(A_{J_k} x^k - b_{J_k})^T \diag \left( \bar{\omega}_i^{k},i \in J_k \right) (A_{J_k} x^k - b_{J_k})}{\|A_{J_k}^T \diag \left( \bar{\omega}_i^{k},i \in J_k \right) (A_{J_k} x^k - b_{J_k}) \|^2} \\
& \geq \frac{1}{\lambda_{\max} \left(  \diag \left( \sqrt{\bar{\omega}_i^{k}},i \in J_k \right)  A_{J_k}  A_{J_k}^T  \diag \left( \sqrt{\bar{\omega}_i^{k}},i \in J_k \right)  \right)}  \\
& = \frac{1}{\lambda_{\max} \left(  A_{J_k}^T \diag \left( \bar{\omega}_i^{k},i \in J_k \right) A_{J_k} \right)}.
\end{align*}
This inequality holds trivially for the second choice (case) of $L_k$. Therefore, we can further bound $L_k$ for all the cases as follows:
\begin{align}
\label{eq:ineqLk}
L_k &\geq \frac{1}{\lambda_{\max} \left(  A_{J_k}^T \diag \left( \bar{\omega}_i^{k},i \in J_k \right) A_{J_k} \right)}  \nonumber \\
& \geq \frac{1}{ (\max_{i \in J_k} \omega_i^k) \lambda_{\max} \left(  A_{J_k}^T \diag \left(1/\|a_i\|^2,i \in J_k \right) A_{J_k} \right)} \nonumber  \\
& \geq \frac{1}{ \omega_{\max} \max_{J \sim \textbf{P}} \lambda_{\max} \left(  A_{J}^T \diag \left(1/\|a_i\|^2,i \in J \right) A_{J} \right)} \nonumber \\
& = \frac{1}{ \omega_{\max} \lambda_{\max}^{\text{block}}}.
\end{align}
Using this bound in the recurrence above  we get:
\begin{align*}
\|x^{k+1} - x_{k+1}^{*} \|^2  & \leq  \| x^k - x_{k}^{*}\|^2 -  \delta (2- \delta) \frac{1}{ \omega_{\max} \lambda_{\max}^{\text{block}}} \sum\limits_{i \in J_k} {\omega}_{i}^{k} \frac{(a_i^T x^k - b_i)^2}{\|a_i\|^2}\\
& \leq   \| x^k - x_{k}^{*}\|^2 -  \delta (2- \delta) \frac{\omega_{\min}}{\omega_{\max} \lambda_{\max}^{\text{block}}} \sum\limits_{i \in J_k}  \frac{(a_i^T x^k - b_i)^2}{\|a_i\|^2}\\
& =   \| x^k - x_{k}^{*}\|^2 -  \delta (2- \delta) \frac{\omega_{\min}}{\omega_{\max} \lambda_{\max}^{\text{block}}}  (x^k - x_{k}^{*})^T \sum\limits_{i \in J_k}  \frac{a_i a_i^T}{\|a_i\|^2} (x^k - x_{k}^{*}).
\end{align*}
Taking now the conditional expectation and using again \eqref{eq:pr_prob},  we get:
\begin{align*}
&\ExpJ{\|x^{k+1} - x_{k+1}^{*} \|^2 | {\cal F}_k} \\
& \leq  \| x^k - x_{k}^{*}\|^2 -  \delta (2- \delta) \frac{\omega_{\min}}{\omega_{\max} \lambda_{\max}^{\text{block}}}   (x^k - x_{k}^{*})^T  A^T \diag \left( \frac{p_i}{\|a_i\|^2},i \in [m]\right) A  (x^k - x_{k}^{*})\\
& = \| x^k - x_{k}^{*}\|^2 -  \delta (2- \delta) \frac{\omega_{\min}}{\omega_{\max} \lambda_{\max}^{\text{block}}}   (x^k - x_{k}^{*})^T  W  (x^k - x_{k}^{*})
\end{align*}
It is also known from the Courant-Fischer theorem  that for any matrix $A$ we have $\|A x \|^2 \geq \lambda_{\min}^{\text{nz}}(AA^T) \|x\|^2$ for all $x \in \text{range}(A^T)$. Moreover, we have  that $x - \Pi_{\cX}(x) \in \text{range}(A^T)$ for any $x$. In conclusion, using that  $W = A^T D A$, with the diagonal matrix $D = \diag \left( \frac{p_i}{\|a_i\|^2},i \in [m]\right)$ invertible, we get that:
\begin{align*}
(x^k - x_{k}^*)^T W (x^k - x_{k}^*) &=  \|D^{1/2} A (x^k - x_{k}^*) \|^2 \geq \lambda_{\min}^{\text{nz}}(A^T D A) \|x^k - x_{k}^*\|^2\\
&=\lambda_{\min}^{\text{nz}}(W) \|x^k - x_{k}^*\|^2.
\end{align*}
Using this inequality in the recurrence above and taking expectation over the entire history we get:
\begin{align*}
\Exp{\|x^{k+1} - x_{k+1}^{*} \|^2} & \leq \left(1  -  \delta (2- \delta) \frac{\omega_{\min} }{\omega_{\max} \lambda_{\max}^{\text{block}}} \lambda_{\min}^{\text{nz}}(W) \right) \Exp{\| x^k - x_{k}^{*}\|^2},
\end{align*}
hence proving the statement of the theorem.
\end{proof}

\noindent There is a tight connection between the constant stepsize \eqref{eq:step_c}  and the adaptive stepsize \eqref{adaptstep}. Indeed, for simplicity let us consider uniform weights $\omega_i^k=1/\tau$ and normalized matrices ($\|a_i\| = 1$ for all $i, k$). Then, from  \eqref{eq:ineqLk}  we obtain:
\[ L_k  =  \tau \frac{\| A_{J_k} x^k  - b_{J_k} \|^2}{ \|  A_{J_k}^T (A_{J_k} x^k  - b_{J_k})  \|^2}  \geq  \tau \frac{1}{\lambda_{\max} \left(  A_{J_k}^T  A_{J_k} \right)}  \geq  \tau  \frac{1}{ \lambda_{\max}^{\text{block}}}. \]
Hence, $L_k$ represents an online approximation of $\tau/\lambda_{\max}^{\text{block}}$ and therefore:
\begin{align}
\label{eq:comp_ca}
\alpha_k =  2 L_k \geq \alpha =  2 \frac{\omega_{\min}}{\omega_{\max}^2 \lambda_{\max}^{\text{block}}} =  2 \tau \frac{1}{ \lambda_{\max}^{\text{block}}}.
\end{align}
In conclusion, the adaptive stepsize \eqref{adaptstep} can be viewed as a practical online approximation of the constant extrapolated stepsize \eqref{eq:step_c}.  Finally, let us simplify the convergence rate \eqref{eq:convrate2}  for the uniform and partition sampling examples of Section \ref{sec:prob} having all the blocks  sampling  the same size  $|J| = \tau$. In this case  we have $p_i = \frac{\tau}{m}$. Let us also consider the particular choices  $\delta=1$, weights $\omega_i=1/\tau$, and  normalized matrices $A$.  Then,   our convergence rate \eqref{eq:convrate2} becomes:
\begin{align}
\label{eq:converate22}
\Exp{\|x^{k} - x_{k}^*\|^2}  \leq  \left(  1 - \frac{\tau}{\lambda_{\max}^{\text{block}}} \frac{\lambda_{\min}^{\text{nz}}(A^TA)}{m}  \right)^k \| x^0 - x_{0}^*\|^2.
\end{align}
We observe that this convergence rate coincides with \eqref{eq:converate11}. However, the adaptive block Kaczmarz scheme has more chances to accelerate, since from \eqref{eq:comp_ca}  the variable stepsize is, in general, larger than the constant stepsize counterpart.

%%%%%%%%%%%%%%%%%%%%%%%%%%%%%%%%%%%%%%%%%%

\subsection{When block Kaczmarz works?}
\label{sec:pave}
Comparing the convergence rates of RBK algorithm  with  the constant stepsize \eqref{eq:step_c}  and with the adaptive stepsize  \eqref{adaptstep}  given in  \eqref{eq:converate11} and \eqref{eq:converate22}, respectively,  with the convergence rate  of the basic Kaczmarz method given in \eqref{eq:ratebasic},  we obtain an improvement  $\frac{\tau}{\lambda_{\max}^{\text{block}}} >1$ for the block variants. Recall that the stochastic conditioning parameter $\lambda_{\max}^{\text{block}}$ is defined as:
\[    \lambda_{\max}^{\text{block}} = \max_{J \sim \textbf{P}} \lambda_{\max} \left( A_J^T \diag \left( \frac{1}{\|a_i\|^2},i \in J\right) A_J \right).  \]
Therefore,  we can get a speed-up even of order approximately $\tau$ for well conditioned matrices, i.e. for matrices having $\lambda_{\max}^{\text{block}} \ll \tau$.  This shows that the probability $\textbf{P}$ plays a key role in defining the importance sampling procedure and consequently in the convergence behavior of RBK.  Fortunately, the operator theory literature provides detailed information about the existence of such good probabilities defining the importance sampling. This is usually referred in the literature as \textit{good paving} \cite{NeeTro:14a}.  This section summarizes the main results from the literature on row paving and provides a technique for constructing a good paving. The idea is to find a random partition of   the rows of the matrix $A$ such that each subset has approximately equal size. Results on  existence  of good row pavings were derived e.g. in \cite{Tro:09}:
\begin{lema}
Let $A$ be  a  normalized   matrix  with  $m$  rows and $\theta \in (0,1)$. Then, there is  a randomized partition $\{J_1,\cdots\!,J_\ell \}$  of the rows indices with $\ell \geq {\cal O}(\| A \|^2 \log(1+m)/\theta^2) $ such that  $ \lambda_{\max}^{\text{block}}  \leq 1+\theta $. 
\end{lema}
Although this is only an  existential result, the literature describes several efficient algorithms for constructing good row pavings. For example, assume that $\kappa$ is a permutation of the set $[m]= \{1, 2,\cdots, m\}$, chosen uniformly at random. For each $i=1:\ell$, define the subsets:
\[  J_i = \left\{  \kappa(l): l= \lfloor (i-1)\frac{m}{\ell} \rfloor + 1, \cdots, \lfloor \frac{m}{\ell} \rfloor \right\}.  \]
It is clear that $\{J_1,\cdots, J_\ell  \}$ is a random partition of $[m]$ into $\ell$ blocks of approximately equal size.  For  every  normalized  matrix,   such a  random  partition  leads to a  row paving  whose  $\lambda_{\max}^{\text{block}}$ is  relatively small.

\begin{lema}
\label{lema:paving}
Let $A$ be  a  normalized   matrix  with  $m$  rows. Consider a randomized partition $\{J_1,\cdots\!,J_\ell \}$  of the rows indices with $\ell \geq \| A \|^2$ subsets. Then, $\{J_1,\cdots, J_\ell \}$  is a row paving with the  upper bound  $ \lambda_{\max}^{\text{block}}  \leq 6 \log (1+m) $ with probability at least $1 - m^{-1}$.
\end{lema}

\noindent A  proof of this type of result appears in  \cite{Tro:11}, see also \cite{NeeTro:14a}.  By merging our  theorems on the convergence of RBK algorithm with the previous result on the good paving, we obtain:

\begin{theorem}
Let $A$ be a normalized matrix and $\{J_1,\cdots, J_\ell  \}$ be a random partition of the rows of $A$, as given by Lemma  \ref{lema:paving}, such that $\tau= m/\ell$ is a positive integer.  Under the assumptions of Theorems \ref{th1:convergence} and \ref{th2:convergence}, the randomized block Kaczmarz method, Algorithm \ref{alg:RBK}, with weights $\omega_i^k = 1/\tau =\ell/m$ for all $i, k$, and constant stepsize \eqref{eq:step_c} or  adaptive stepsize  \eqref{adaptstep} with $\delta=1$, admits the convergence estimate:
\begin{align}
\label{eq:convrate_paving}
 \Exp{\|x^{k} - x_{k}^*\|^2}  \leq  \left(  1 -  \frac{\lambda_{\min}^{\text{nz}}(A^TA)}{ 6 \log(1+m) \|A\|^2} \right)^k  \| x^0 - x_{0}^*\|^2.
\end{align}
\end{theorem}

\noindent In conclusion,  our new convergence  analysis shows when a block variant of Kaczmarz algorithm really works, i.e. we can choose a subset of rows $\tau > 1$ at each step, when $\lambda_{\max}^{\text{block}}  \ll \tau$. Hence,  a distributed implementation of the RBK algorithm  is most effective when the probability distribution $\textbf{P}$ yields  a partition of the rows into well-conditioned blocks.  Otherwise, we can just apply the basic Kaczmarz algorithm with $\tau=1$.  Moreover, our analysis shows that  the \textit{optimal batchsize} is of order $\tau \sim m/\|A\|^2$.  Assuming, for simplicity,  that $\tau= m/\ell$ is a positive integer, from  Lemma  \ref{lema:paving}
\[ \lambda_{\max}^{\text{block}}  \leq 6 \log(1+m)  \ll  \tau  = \frac{m}{\ell}  \simeq \frac{m}{\|A\|^2}   \]
holds with high probability, provided that the matrix $A$ satisfies the following inequality
\begin{align}
\label{eq:goodmatrix}
 \|A\|^2 \ll \frac{m}{6 \log(1+m)}.
\end{align}
Recall that, for a  normalized matrix $A$ with $m$ rows, the  squared  spectral  norm $\|A\|^2$ attains    its  maximal  value $m$ when  $\text{rank}(A) =1$, i.e.  its  rows  are  identical.   Therefore, the inequality \eqref{eq:goodmatrix} stipulates that the rows of $A$ must exhibit a large amount of diversity in order for  RBK algorithm with extrapolated stepsizes  \eqref{eq:step_c} or  \eqref{adaptstep} to perform better than the basic Kaczmarz scheme.  Note that  convergence rates similar to \eqref{eq:convrate_paving} has been derived in \cite{NeeTro:14a} for the block projection  Kaczmarz algorithm \eqref{eq:fullK} with the particular stepsize  $\alpha_k=1$.  However,  RBK requires the computation of $\tau$ scalar products in $\rset^n$ at each iteration, so that its  computational  cost per iteration  is  ${\cal O}(\tau n)$, and thus cheaper than the one corresponding to block projection  Kaczmarz  \eqref{eq:fullK} that requires solving a least-squares problem at each iteration in about ${\cal O}(\tau^2 n)$.

%\vspace{0.2cm}
%\noindent \textcolor{red}{Comparing to basic Kaczmarz.....in terms of total cost....??????????????}

%%%%%%%%%%%%%%%%%%%%%%%%%%%%%%%%%%%%%%%%%%%
%%%%%%%%%%%%%%%%%%%%%%%%%%%%%%%%%%%%%%%%%%

\section{Randomized block Kaczmarz algorithm with  Chebychev-based stepsize}
\label{sec:KC}
Finally, we  show that we can also choose extrapolated stepsizes  in RBK (Algorithm \ref{alg:RBK}) based on the roots of Chebyshev polynomials.  For simplicity, we consider either the uniform or partition sampling of Section \ref{sec:prob} having $|J| = \tau$.  We also assume normalized matrices $A$ and constant weights   $\omega_k^i = 1/\tau$ for all $k, i$. Under these settings, for   RBK algorithm  with Chebyshev-based stepsize we derive  linear or sublinear convergence estimates  depending whether  $\lambda_{\min}(AA^T) \!>\! 0$ or $\lambda_{\min}(AA^T) \!=\! 0$, respectively. Below we investigate these two cases.

\subsection{Case 1:  $\lambda_{\min}(AA^T) >0$} We get the following linear convergence for this variant of  RBK:

\begin{theorem}
\label{th3:convergence}
Assume normalized matrix $A$ such that $\lambda_{\min}(AA^T) >0$. Let $\{ x^k \}_{k \ge 0}$ be generated by  RBK (Algorithm \ref{alg:RBK})  with the uniform or partition sampling and  the weights $\omega_k^i = 1/\tau$ for all $k, i$. Further,  for a fixed number of iterations $k$ the  stepsizes  $\{\alpha_j\}_{j=0}^{k-1}$ are depending on the roots of the Chebyshev polynomial of degree $k$ (see Appendix)  as follows:
\[  \alpha_j = \frac{2m}{\left( \lambda_{\max}(AA^T) + \lambda_{\min}(AA^T) \right)  +  \left(\lambda_{\max}(AA^T) - \lambda_{\min}(AA^T)\right) \cos \left(  \frac{2 \kappa(j)+1}{2k} \pi \right)},  \]
where $\kappa$ is a  permutation of $[0\!:\!k\!-\!1]$. Then, we have the following linear convergence for  expected iterates:
\begin{align}
\label{eq:convrate3}
\| \Exp{x^{k} - x_k^{*}} \|^2 & \leq    \frac{ 4  \lambda_{\max}(AA^T)}{\lambda_{\min}(AA^T) } \left( 1 - \sqrt{\frac{\lambda_{\min}(AA^T)}{ \lambda_{\max}(AA^T) }} \right)^{2k}  \| x^0 - x_0^* \|^2.
\end{align}
\end{theorem}

\begin{proof}
For the iteration of   RBK (Algorithm \ref{alg:RBK})  we have for any solution $x^* \in \cX$:
\begin{align*}
x^{k+1} - x^* & =  x^k - x^* - \alpha_k   \left(\sum\limits_{i \in J_k} \omega_k^i \frac{a_i^T x^k - b_i}{\|a_i\|^2} a_i \right) \overset{\omega_k^i=1/\tau, \|a_i\|=1}{=} x^k - x^* -  \frac{\alpha_k}{\tau}   \left(\sum\limits_{i \in J_k}  (a_i^T x^k - b_i) a_i \right)\\
& = x^k - x^* - \frac{\alpha_k}{\tau}  \left( \sum\limits_{i \in J_k}   a_i a_i^T (x^k - x^*) \right)  = \left(I_n - \frac{\alpha_k}{\tau} \left( \sum\limits_{i \in J_k}   a_i a_i^T \right)  \right) (x^k - x^*).
\end{align*}
Taking conditional  expectation and using \eqref{eq:pr_prob} with $p_i = \tau/m$ for uniform or partition sampling, we get:
\begin{align}
\label{eq:recbasic}
\ExpJ{x^{k+1} - x^* | {\cal F}_k} &  =   \left(I_n - \frac{\alpha_k}{\tau}  \left( \sum\limits_{i \in [m]} p_i  a_i a_i^T \right)  \right) (x^k - x^*) \\
& \overset{p_i=\tau/m}{=} \left(I_n -  \frac{\alpha_k}{m}  A^T  A \right) (x^k - x^*).   \nonumber
\end{align}
Multiplying from the left this recurrence with $A$ we get:
\begin{align*}
\ExpJ{Ax^{k+1} - Ax^* | {\cal F}_k} & =  \left(A -  \frac{\alpha_k}{m}  A A^T A \right) (x^k - x^*)   = \left(I_m -  \frac{\alpha_k}{m} AA^T \right) (Ax^k - Ax^*)
\end{align*}
or equivalently, using that $Ax^*=b$ and taking expectations over the entire history, we obtain:
\[  \Exp{Ax^{k+1} - b}  =   \left(I_m -  \frac{\alpha_k}{m} AA^T \right) \Exp{Ax^k - b}.  \]
Iterating this recurrence and defining the matrix  $G= \frac{1}{m} AA^T \in \rset^{m \times m}$ we obtain:
\[  \Exp{Ax^{k} - b}  =  \prod_{j=0}^{k-1} \left(I_m - \alpha_j \frac{1}{m} AA^T \right) (Ax^0 - b) =  \prod_{j=0}^{k-1} \left(I_m - \alpha_j G \right) (Ax^0 - b).  \]
If we define the polynomial in the matrix $G$ as $P_k(G) = \prod_{j=0}^{k-1} \left(I_m - \alpha_j G \right)$, then we can bound the norm of the  expected residual by:
\[   \| \Exp{Ax^{k} - b} \| =  \|  P_k(G) (Ax^0 - b) \| \leq \|P_k(G)\| \cdot  \| Ax^0 - b \|.  \]
Recall that we consider consistent  linear system with  $\lambda_{\min}(AA^T)  >0$.  Then, from standard reasoning  the spectrum of $G =  \frac{1}{m} AA^T$ satisfies $\Lambda(G)  \subset \rset^{}_{++}$. More precisely:
\[   0<  \underbrace{\frac{1}{m}  \lambda_{\min}(AA^T)}_{=\ell}   \leq  \lambda_i(G) \leq   \underbrace{\frac{1}{m} \lambda_{\max}(AA^T)}_{= u} < \infty \quad \forall i=1:m.  \]
Therefore, if we denote by $\lambda_i$ the $i$th eigenvalue of $G$, we have the following bound:
\begin{align*}
 \| \Exp{Ax^{k} - b} \|  & \leq \|P_k(G)\| \cdot  \| Ax^0 - b \| \leq \max_{i=1:m} |P_k(\lambda_i)| \cdot  \| Ax^0 - b \|  \leq \max_{\lambda \in [\ell, u]} |P_k(\lambda)| \cdot  \| Ax^0 - b \|.
\end{align*}
In conclusion, we can choose the stepsizes $\alpha_j$ for $j=0:k-1$ such that $P_k(\lambda) =  \prod_{j=0}^{k-1} \left(1 - \alpha_j \lambda \right)$ is  the polynomial least deviating from zero on the interval $[\ell, u]$ and satisfying $P_k(0)=1$. It is well known that this is the polynomial given in terms of a Chebyshev  polynomial (see Appendix for a brief review of the main properties of Chebyshev polynomials):
\[   P_k(\lambda) = T_k\left(  \frac{2 \lambda}{u - \ell} - \frac{u + \ell}{u - \ell}  \right) \big{/} T_k \left(  - \frac{u + \ell}{u - \ell} \right).  \]
Then, we can guarantee the following linear convergence in expectation (see Lemma \ref{lemma:cheb_appendix} in Appendix):
\begin{align}
\label{eq:linconv_acc}
\| \Exp{Ax^{k} - b} \|  \leq  2 \left( \frac{\sqrt{u} - \sqrt{\ell}}{\sqrt{u} + \sqrt{\ell}} \right)^k  \| Ax^0 - b \| \leq 2  \left( 1 - \sqrt{\frac{\lambda_{\min}(AA^T)}{\lambda_{\max}(AA^T)}} \right)^k  \| Ax^0 - b \|.
\end{align}
The stepsizes $\alpha_j$,  for $j=0:k-1$,  are chosen as the inverse roots of polynomial  $ P_k(\lambda)$ (see Appendix):
\begin{align*}
\alpha_j & = 2/ \left(  (u +\ell) + (u - \ell) \cos \left(  \frac{2 \kappa(j)+1}{2k} \pi \right)   \right)\\
& = 2m/ \left( (\lambda_{\max}(AA^T) + \lambda_{\min}(AA^T))  + (\lambda_{\max}(AA^T) - \lambda_{\min}(AA^T)) \cos \left(  \frac{2 \kappa(j)+1}{2k} \pi \right)  \right),
\end{align*}
where $\kappa$ is some fixed permutation of $[0:k-1]$.  We can also   derive convergence rates in $ \Exp{x^{k} - x_k^{*}}$ using  that $\Exp{x^{k} - x_k^{*}} \in \text{range}(A^T)$, and consequently from Courant-Fischer lemma and  \eqref{eq:linconv_acc} we have:
\begin{align*}
& \lambda_{\min}(AA^T)  \| \Exp{x^{k} - x_k^{*}} \|^2  \leq  \| A \Exp{x^{k} - x_k^*} \|^2 =   \| \Exp{Ax^{k} - b} \|^2 \\
& \leq  4  \left( 1 - \sqrt{\frac{\lambda_{\min}(AA^T)}{\lambda_{\max}(AA^T)}} \right)^{2k}  \| Ax^0 - b \|^2 \leq    4  \lambda_{\max}(AA^T)  \left( 1 - \sqrt{\frac{\lambda_{\min}(AA^T)}{\lambda_{\max}(AA^T)}} \right)^{2k}  \| x^0 - x_0^* \|^2.
\end{align*}
proving thus the linear convergence estimate   of the theorem.
\end{proof}

\noindent From Jensen's inequality we  have  $\| \Exp{\cdot} \| \leq \Exp{\| \cdot\|}$. In conclusion,  $\| \Exp{\cdot} \| $ is a weaker criterion than $ \Exp{\| \cdot\|}$. Note that convergence rates in the weaker criterion  $\|\Exp{x^k - x_k^*}\|$ have been also given  for another  variant of  Kaczmarz algorithm in \cite{RicTak:17} or for the random coordinate descent method in \cite{SunYe:17}. Moreover, the convergence rate from Theorem \ref{th3:convergence} is the same as for the conjugate gradient method  and it is optimal for this class of iterative schemes.  However, since this rate does not depend on the size of the blocks $|J|$, then we usually implement this accelerated variant of Kaczmarz by sampling single rows, that is,  $|J|=1$.

\subsection{Case 2:  $\lambda_{\min}(AA^T) =0$} In this case  we get  sublinear convergence for this variant of  RBK:

\begin{theorem}
\label{th4:convergence}
Assume normalized matrix $A$ such that $\lambda_{\min}(AA^T) =0$. Let $\{ x^k \}_{k \ge 0}$ be generated by  RBK (Algorithm \ref{alg:RBK})  with the uniform or partition sampling and  the weights $\omega_k^i = 1/\tau$ for all $k, i$.   Further,  for a fixed number of iterations $k$ the  stepsizes  $\{\alpha_j\}_{j=0}^{k-1}$ are depending on the roots of the Chebyshev polynomial of degree $k$  as follows:
\[  \alpha_j =  \frac{m \left( 1 - \cos \left( \frac{2k+1}{2(k+1)} \pi \right) \right) }{  \lambda_{\max}(AA^T)   \left( \cos \left(  \frac{2 \kappa(j)+1}{2(k+1)} \pi \right) - \cos \left(  \frac{2k+1}{2(k+1)} \pi \right)  \right) },  \]
where $\kappa$ is some  permutation of $[0\!:\!k-1]$. Then, we have the following sublinear convergence for the residual of the normal system in expectation:
\begin{align}
\label{eq:convrate4}
\|\Exp{  A^T A x^k - A^T b }  \|  =  \|\Exp{   A x^k -  b }  \|_{(AA^T)} \leq   \frac{ \pi \lambda_{\max}(AA^T) }{ 2(k+1)^2}   \| x^0 - x^* \|.
\end{align}
\end{theorem}

\begin{proof}
From \eqref{eq:recbasic} we also get the relation:
\begin{align*}
\Exp{x^{k} - x^*} &  =   \prod_{j=0}^{k-1} \left(I_n - \alpha_j \frac{1}{m}  A^T  A \right) (x^0 - x^*).
\end{align*}
Now, if we consider the normal system $A^TA x=  A^T b$, which coincides with $\nabla f(x) = 0$,   we have:
\begin{align*}
\|  \Exp{  A^T A x^k - A^T b }  \| = \|  \Exp{  A^T A (x^k - x^*) }  \| = \|  A^T A  \prod_{j=0}^{k-1} \left(I_n - \alpha_j \frac{1}{m}  A^T  A \right) (x^0 - x^*) \|,
\end{align*}
where $x^*$ denotes  any solution of $Ax=b$ (recall that we consider consistent  linear systems). If we define the matrix $G = \frac{1}{m}  A^T  A$ and the polynomial  $Q_k(G) = G \prod_{j=0}^{k-1} \left(I_n - \alpha_j G \right)$, then we  obtain the following bound for the residual of the normal system in expectation:
\[  \|  \Exp{  A^T A x^k - A^T b }  \| =   m \|  Q_k(G) (x^0 - x^*)  \|  \leq  m \|  Q_k(G) \|  \|x^0 - x^*\|.   \]
Since  we assume $\lambda_{\min}(AA^T) = \lambda_{\min}(A^TA)= 0$,  then the spectrum of $G =  \frac{1}{m} A^TA$ satisfies:
\[   0   \leq  \lambda_i(G)  \leq   \underbrace{\frac{1}{m} \lambda_{\max}(A^TA)}_{= u} < \infty \quad \forall i=1:m.  \]
Therefore, if we denote by $\lambda_i$ the $i$th eigenvalue of $G$, we have the following bound:
\begin{align*}
 \|\Exp{  A^T A x^k - A^T b }  \|  & \leq m  \|Q_k(G)\| \cdot  \| x^0 - x^* \|  \leq m  \max_{i=1:m} |Q_k(\lambda_i)| \cdot  \| x^0 - x^* \|  \\
& \leq m  \max_{\lambda \in [0, u]} |Q_k(\lambda)| \cdot  \| x^0 - x^* \|.
\end{align*}
In conclusion, we can choose the stepsizes $\alpha_j$ for $j=0:k-1$ such that $Q_k(\lambda) = \lambda  \prod_{j=0}^{k-1} \left(1 - \alpha_j \lambda \right)$ of degree $k+1$ is  the polynomial least deviating from zero on the interval $[0, u]$ and satisfying $Q_k(0)=0$ and $Q_k'(0) = 1$. We show below that this  polynomial is also given in terms of a Chebyshev  polynomial.   Indeed, let us consider the closest root to $-1$ of the Chebyshev polynomial of degree $k+1$ (i.e. $T_{k+1}$):
\[   r_{k+1} = \cos \left(  \frac{2k+1}{2(k+1)} \pi \right)   = \cos \left(   \pi - \frac{1}{2(k+1)} \pi \right).  \]
Then, we define the polynomial:
\[  Q_k(\lambda) = \frac{u}{1-r_{k+1}} \frac{ T_{k+1} \left(  r_{k+1} + \frac{1-r_{k+1}}{u} \lambda \right) }{T_{k+1}'(r_{k+1})}.    \]
Note that this polynomial satisfies the required properties:  $\text{deg}(Q_k) = k+1$, $Q_k(0) = \frac{u  T_{k+1} (r_{k+1})}{(1-r_{k+1}) T_{k+1}'(r_{k+1})} =0$ (recall that $r_{k+1}$ is the $k+1$ root of $T_{k+1}$) and   $Q_k'(0) = \frac{T_{k+1}'(r_{k+1})}{T_{k+1}'(r_{k+1})} = 1$. In conclusion, we get  the following bound for this choice of $Q_k(\lambda)$:
\begin{align*}
m  \max_{\lambda \in [0, u]}  |Q_k(\lambda)|  =  m  \max_{\lambda \in [0, u]}  |   \frac{u}{1-r_{k+1}} \frac{ T_{k+1} \left(  r_{k+1} + \frac{1-r_{k+1}}{u} \lambda \right) }{T_{k+1}'(r_{k+1})} |  \leq  m  \frac{u}{| T_{k+1}'(r_{k+1}) |}     =     \frac{ \lambda_{\max}(A^TA)}{| T_{k+1}'(r_{k+1}) |},
\end{align*}
where in the inequality we used that $|T_{k+1}(x)| \leq 1$ for any $x \in [-1, \; 1]$  and that the root  $r_{k+1} \leq 0$ (see Appendix). Further, since $T_{k+1} (\cos(\theta)) = \cos ((k+1) \theta)$, if we differentiate we  get $\sin(\theta) T_{k+1}'  (\cos(\theta)) = (k+1) \sin((k+1) \theta)$. Now, for $ r_{k+1} = \cos \left(   \pi - \pi/(2k+2)  \right)$ we obtain:
\[  | T_{k+1}'(r_{k+1})|  = \frac{(k+1) |\sin((k+1) \pi - \pi/2)|}{|\sin( \pi - \pi/(2k+2)  )|} =  \frac{k+1}{|\sin( \pi - \pi/(2k+2)  )|} = \frac{2(k+1)^2}{\pi},  \]
for $k$ sufficiently large (we used that $\sin(\pi-\theta) \sim \theta$ for $\theta$ small). In conclusion, we get the following sublinear convergence (using the notation $\|u\|_{(AA^T)} = \|A^Tu\| $):
\[  \|\Exp{  A^T A x^k - A^T b }  \|  =  \|\Exp{   A x^k -  b }  \|_{(AA^T)}   \leq   \frac{ \pi \lambda_{\max}(A^TA) }{ 2(k+1)^2}   \| x^0 - x^* \|,  \]
for $k$ sufficiently large (i.e. for $k$ such that $\sin( \pi - \pi/(2k+2)) \sim  \pi/(2k+2)$). Finally, using that $\lambda_{\max}(A^TA) = \lambda_{\max}(AA^T)$ we get \eqref{eq:convrate4}.  The stepsizes $\alpha_j$,  for $j=0:k-1$,  are chosen as the inverse roots of polynomial  $ Q_k(\lambda)$ (see Appendix):
\begin{align*}
\alpha_j & \!=\! (1 - r_{k+1}) u^{-1} / \left( \! \cos \left(  \frac{2 \kappa(j)+1}{2(k+1)} \pi \! \right) - r_{k+1} \right)   \!=\!  \frac{m \left( 1 - \cos \left( \frac{2k+1}{2(k+1)} \pi \right) \right) }{  \lambda_{\max}(AA^T)   \left( \cos \left(  \frac{2 \kappa(j)+1}{2(k+1)} \pi \right) - \cos \left(  \frac{2k+1}{2(k+1)} \pi \right)  \right) },
\end{align*}
where $\kappa$ is some fixed permutation of $[0:k-1]$.
\end{proof}

\noindent Note that the RBK algorithm with Chebyshev-based stepsize  belongs to the class of Chebyshev semi-iterative methods \cite{GolVar:61}. However, from our knowledge, this work is the first one that uses the properties of the Chebyshev polynomials  in order to accelerated the convergence rate of randomized block Kaczmarz (RBK) algorithm.  Other types of acceleration of Kaczmarz algorithm have been proposed  e.g. in \cite{HanNie:90,LiuWri:16,RicTak:17}.  For example, in  \cite{RicTak:17}  two dependent  steps of basic randomized Kaczmarz algorithm are taken, one from $x^k$ and one from $x^{k-1}$, and then an affine combination of the results produces the next iterate $x^{k+1}$. For this scheme, \cite{RicTak:17} derives a similar convergence rate as in Theorem \ref{th3:convergence}. In  \cite{LiuWri:16} Nesterov's accelerated random coordinate descent method from \cite{Nes:12}  is applied to the dual problem \eqref{eq:lsdual}, leading in the primal space to an accelerated randomized Kaczmarz scheme with momentum. For this  accelerated Kaczmarz scheme \cite{LiuWri:16} derives   the  convergence rate $\Exp{\|x^k - x_k^*\|^2} \leq (1 - \sqrt{\lambda_{\min}(AA^T)}/m)^k \|x^0 - x_0^*\|^2$.  Although this rate is worse than  \eqref{eq:convrate3} in terms of constants, it is given in the stronger criterion $\Exp{\|x^k - x_k^*\|^2} $.  Remains an open problem whether Theorem \ref{th3:convergence} can be also given  in the stronger criterion $\Exp{\|x^k - x_k^*\|^2} $.

%\begin{remark}
%In paper \cite{SunYe:17} the authors also provide convergence for random coordinate descent in the weak criterion $\|\Exp{x^k} - x^*\|$. However, in paper \cite{LeeWri:16} the authors are able to extend the analysis from  \cite{SunYe:17} and fix this issue providing convergence of random coordinate descent in  the stronger criterion $\Exp{\|x^k - x^*\|}$.
%\end{remark}

%%%%%%%%%%%%%%%%%%%%%%%%%%%%%%%%5
%%%%%%%%%%%%%%%%%%%%%%%%%%%%%%%%%

%\section{Simulations}
%\label{sec:num}
%\noindent First we considered a uniformly random generated linear system $Ax = b$ with $A \in \rset^{1000 \times 100}$. We generated the data using plain
%from Matlab. We took $X_i = \{x \; | \; A_ix = b_i\}$, which means that each linear equation denotes a set and we set $N=50$. For all algorithms the random indices are generated using distribution $P(i) = \frac{\norm{A_i}^2}{\norm{A}_F^2}$. In figure 1 we observe no significant improvement of adaptive scheme over the constant stepsize method.

%%%%%%%%%%%%%%%%%%%%%%%%%%%%%%%%%%%%%

%%%%%%%%%%%%%%%%%%%%%%%%%%%%%%%%%%%

\section*{Appendix  (Chebyshev polynomials)}
\label{apendix}
In this section some properties of the Chebyshev polynomials are briefly reviewed. We refer to e.g. \cite{OlsTyr:14} for more details on Chebyshev polynomials. The Chebyshev polynomials  $T_k(x)$, where   $\text{deg} (T_k) = k$  and $k \geq 0$,  are defined by the recursive relation:
\[  T_0(x) = 1, \;\;  T_1(x) = 1,  \;\;  T_{k+1}(x)  = 2x T_k(x) - T_{k-1}(x).   \]
From the above recurrence we observe that the leading coefficient of $T_k(x)$ is $2^{k-1}$, i.e.  $T_k(x)  = 2^{k-1} x^k + \; \text{lower powers of} \; x$.  In particular, for $x \in [-1,\; 1]$, the  Chebyshev polynomials can be written  equivalently:
\[  T_k(x)  =  \cos(k \arccos (x)).   \]
The equivalence can be verified as follows using that $x = \cos(\theta)$:
\begin{align*}
T_k(x)  & =  2x \cos((k-1) \arccos (x)) - \cos((k-2) \arccos (x))  = 2 \cos(\theta) \cos((k-1)\theta) - \cos((k-2)\theta) \\
& = \cos(k \theta) + \cos((k-2)\theta) - \cos((k-2)\theta) = \cos(k \theta) =  \cos(k \arccos (x)).
\end{align*}
It follows that $T_k(1)=1$. From this representation of $T_k(x)$ it also follows that:
\[  \max_{x \in [-1, \; 1]} |T_k(x)| = 1.  \]
Moreover, all the $k$ roots of $T_k(x)$ are given by:
\[   x_i = \cos \left(   \frac{2i-1}{2k} \pi \right) \quad \text{for} \;\;  i=1:k. \]
In conclusion, we get also the following representation for $T_k(x)$:
\[  T_k(x)  = 2^{k-1} \cdot \prod_{i=1}^k  \left(  x - \cos \left(   \frac{2i-1}{2k} \pi \right) \right).  \]
It is also easy to see the following interval transformation $[\ell, \; u] \to [-1, \; 1]$ through the relation:
\[  -1 \leq  \frac{2x}{u-\ell} - \frac{u+\ell}{u-\ell} \leq 1  \quad \text{for} \;\;  \ell \leq x \leq u.    \]
One important property of the Chebyshev polynomials is that $\frac{1}{2^{k-1}}T_k(x)$ has minimal deviation from $0$ among all polynomials of degree $k$ with leading coefficient $1$ on $[-1, \; 1]$:
\begin{align}
\label{eq:minC}
\max_{x \in [-1, \; 1]}  \frac{1}{2^{k-1}}|T_k(x)|  \leq \max_{x \in [-1, \; 1]} | P_k(x)| \quad \forall  P_k(x) \;  \text{with leading coefficient} \; 1 \; \text{and} \;   \text{deg}(P_k) = k.
\end{align}
An immediate  consequence of the above property valid for Chebyshev polynomials  is the following lemma:
\begin{lemma}
\label{lemma:cheb_appendix}
Let $0 < \ell <  u$ and $T_k^{(\ell,u)}(x) = T_k \left(  \frac{2x}{u-\ell} - \frac{u+\ell}{u-\ell} \right)$. Then,  the optimal value and the optimal polynomial  $P_k^*$ of the following optimization problem are:
\[  \min_{P_k(x): \;\text{deg}(P_k) = k, P_k(0)=1} \max_{x \in [\ell, \; u]} |P_k(x)| =  \frac{1}{T_k^{(\ell,u)}(0)} \leq 2  \left( \frac{\sqrt{u} - \sqrt{\ell}}{\sqrt{u} + \sqrt{\ell}} \right)^k \quad \text{and} \quad P_k^*(x) = \frac{T_k^{(\ell,u)}(x)}{T_k^{(\ell,u)}(0)}.   \]
\end{lemma}

\end{document}